\documentclass[11pt]{amsart}
\usepackage[T1]{fontenc}
\usepackage[english]{babel}
\usepackage{kpfonts}

\usepackage{amsmath,amssymb,amsthm,enumerate,xspace}
\usepackage{comment}
\usepackage{accents}
\usepackage[colorlinks=true,citecolor=blue,linkcolor=blue]{hyperref}
\usepackage{tikz}
\usepackage{mathrsfs}

\numberwithin{equation}{section}

\newtheorem{thm}{Theorem}[section]

\newtheorem{prop}[thm]{Proposition}
\newtheorem{lem}[thm]{Lemma}
\newtheorem{cor}[thm]{Corollary}
\newtheorem{propdef}[thm]{Proposition/Definition}
\newtheorem*{iprob*}{Problem}

\theoremstyle{definition}

\newtheorem{defi}[thm]{Definition}
\newtheorem{rem}[thm]{Remark}

\newcommand{\ZZ}{\mathbf{Z}}
\newcommand{\RR}{\mathbf{R}}
\newcommand{\NN}{\mathbf{N}}
\newcommand{\CC}{\mathbf{C}}
\newcommand{\QQ}{\mathbf{Q}}
\newcommand{\KK}{\mathbf{K}}

\newcommand{\HR}{\mathbf{H}_{\RR}}
\newcommand{\HRI}{\HR^{\infty}}
\newcommand{\HC}{\mathbf{H}_{\CC}}
\newcommand{\HK}{\mathbf{H}_{\KK}}
\newcommand{\HCI}{\HC^{\infty}}
\newcommand{\HKI}{\HK^{\infty}}
\newcommand{\SO}{\mathbf{SO}}
\newcommand{\SU}{\mathbf{SU}}

\newcommand{\GL}{\mathbf{GL}}
\newcommand{\PSL}{\mathbf{PSL}}
\newcommand{\PGL}{\mathbf{PGL}}
\newcommand{\SSS}{\mathbf{S}}
\newcommand{\OO}{{\mathbf O}}
\newcommand{\UU}{{\mathbf U}}
\newcommand{\tUU}{\widetilde\UU}
\newcommand{\PO}{\mathbf{PO}}
\newcommand{\PU}{\mathbf{PU}}
\newcommand{\Isom}{{\rm Is}}

\newcommand{\Stab}{\mathrm{Stab}} 
\newcommand{\se}{\subseteq}
\newcommand{\inv}{^{-1}}
\newcommand{\ro}{\varrho}
\newcommand{\fhi}{\varphi}

\newcommand{\epsi}{\epsilon}
\newcommand{\lra}{\longrightarrow}

\newcommand{\ol}{\overline}
\newcommand{\cat}[1]{{\upshape CAT({\ensuremath#1})}\xspace}

\DeclareMathOperator{\arcosh}{arcosh}
\newcommand{\her}[2]{\left\langle#1,#2\right\rangle}
\newcommand{\ctran}{^*}
\newcommand{\one}{\boldsymbol{1}}
\DeclareMathOperator{\cartan}{Cart}
\DeclareMathOperator{\Reel}{Re}
\newcommand{\fraisse}{Fra\"\i ss\'e\xspace}
\newcommand{\hc}{H_\mathrm{c}}
\def\hyph{-\penalty0\hskip0pt\relax}

\title[Functions of hyperbolic type]{Notes on functions of hyperbolic type}
\date{July 2018}

\author[Nicolas Monod]{Nicolas Monod}
\address{EPFL, Switzerland}
\email{nicolas.monod@epfl.ch}
\begin{document}
\begin{abstract}
Functions of hyperbolic type encode representations on real or complex hyperbolic spaces, usually infinite\hyph{}dimensional.

These notes set up the complex case. As applications, we prove the existence of a non-trivial deformation family of representations of $\SU(1,n)$ and of its infinite\hyph{}dimensional kin $\Isom(\HCI)$. We further classify all the self\hyph{}representations of $\Isom(\HCI)$ that satisfy a compatibility condition for the subgroup $\Isom(\HRI)$. It turns out in particular that translation lengths and Cartan arguments determine each other for these representations.

In the real case, we revisit earlier results and propose some further constructions.
\end{abstract}
\maketitle
\thispagestyle{empty}

\setcounter{tocdepth}{1}
\tableofcontents
\newpage
\section{Introduction}
\subsection{Context}
A fundamental tool for the study of unitary representations of a group $G$ is the notion of (complex) \emph{function of positive type} on $G$, namely a function $F\colon G\to\CC$ such that
\begin{equation}\label{eq:def:FPT}
\sum_{j,k} \ol{c_j} \,c_k\, F(g_j\inv g_k) \geq 0
\end{equation}
holds for all $n\in \NN$, all $c_1, \ldots, c_n\in \CC$ and all $g_1, \ldots, g_n\in G$. Indeed the GNS construction establishes a perfect correspondance between such functions and unitary representations together with a cyclic vector. The great advantage of this correspondance is that it relates two very differents worlds: representations can be constructed analytically, harnessing e.g.\ the power of general harmonic analysis --- whilst the functions $F$ are concrete first-order structures on $G$, amenable to direct computations.

There is no substantial difference for \emph{orthogonal} representations on \emph{real} Hilbert spaces: a real function of positive type is defined in the same way and it is sufficient to check~\eqref{eq:def:FPT} for $c_j\in \RR$, provided $F$ is symmetric, i.e.\ $F(g\inv) = F(g)$. We recall furthermore that affine isometric actions on Hilbert spaces correspond to functions of \emph{conditionally negative type}~\cite[\S4]{GuichardetLNM}.

\medskip
\itshape
The purpose of these notes is to pursue the study of the corresponding tools for isometric representations on hyperbolic spaces.\upshape\ This study was initiated, in the real case, in~\cite{Monod-Py2_arx}. Just as in the Hilbertian world, our hyperbolic spaces may happen to be finite\hyph{}dimensional but are generally not. We thus write $\HR^\kappa$ for the real hyperbolic space of dimension $\kappa$ (any cardinal); formal definitions are recalled in Section~\ref{sec:prelim}. The separable infinite\hyph{}dimensional case, $\kappa=\aleph_0$, is the most common; we then simply write $\HRI$. We postpone the complex hyperbolic case to later in this introduction --- because it turns out to bring interesting complications.

\subsection{Real hyperbolic kernels}
In the Hilbertian setting, the first step is to forget the group structure and to characterise maps $h\colon X\to H$ from a \emph{set} $X$ to a Hilbert space $H$. This yields a kernel $\Phi\colon X\times X \to\CC$ defined by $\Phi(x,y) = \her xy$ which satisfies
\begin{equation}\label{eq:def:KPT}
\sum_{j,k} \ol{c_j}\, c_k \,\Phi(x_j, x_k)\geq 0.
\end{equation}
Any map $\Phi\colon X\times X \to\CC$ satisfying~\eqref{eq:def:KPT} is called a \emph{kernel of complex positive type}, and the GNS construction then yields $h$ and $H$ as above, with $h(X)$ total in $H$. Again, the only change in the real case is that we need to assume $\Phi$ symmetric (it is automatically Hermitian if~\eqref{eq:def:KPT} holds for complex $c_j$).

\medskip
Turning to hyperbolic spaces, we want to characterise maps $f\colon X\to \HR^\kappa$ purely in terms of the distance function $d$ of $\HR^\kappa$ viewed on $X\times X$. Specifically, $d$ is determined by the Riemannian metric of constant sectional curvature~$-1$ and $\cosh d$ will be the relevant kernel. The following is maybe folklore and can be found in~\cite{Monod-Py2_arx}.

\begin{propdef}\label{propdef:KHT}
Given a set $X$, the following are equivalent for any symmetric function $\beta\colon X\times X\to\RR_{\geq 0}$ taking the constant value~$1$ on the diagonal:
\begin{enumerate}
\item There is a map $f\colon X\to \HR^\kappa$ for some cardinal $\kappa$ such that
\begin{equation*}
\beta(x,y) = \cosh d\big(f(x), f(y)\big)
\end{equation*}
holds for all $x,y \in X$.
\item For every $x_0\in X$, the kernel
\begin{equation*}
(x,y)\longmapsto \beta(x, x_0) \,\beta(x_0, y) -  \beta(x,y)
\end{equation*}
is of positive type.\label{pt:KHT:KPT}
\item For some $x_0\in X$, the kernel above is of positive type.\label{pt:KHT:some}
\end{enumerate}
We then call $\beta$ a  \emph{kernel of real hyperbolic type}.
\end{propdef}

\begin{rem}
In~\cite{Monod-Py2_arx}, this characterisation was introduced using an inequality that more closely reflects the reverse Cauchy--Schwartz inequality for Minkowski spaces, namely
\begin{equation*}
\sum_{j,k=1}^n c_j c_k \beta(x_j, x_k) \leq \Big(\sum_{k=1}^n c_k \beta(x_k, x_0)\Big)^2.
\end{equation*}
This is identical to~\ref{propdef:KHT}\eqref{pt:KHT:KPT} after regrouping terms.
\end{rem}

In analogy with the classical case, a \emph{function of real hyperbolic type} on a group $G$ is a function $F\colon G\to \RR_{\geq 0}$ such that $(g,h)\mapsto F(g\inv h)$ is a kernel of real hyperbolic type. Thus, given an isometric $G$-action on $\HR^\kappa$ and $p\in \HR^\kappa$, the function $F(g) = \cosh d(g p, p)$ is of real hyperbolic type.

The converse passage from functions to group representations is slightly less straightforward than in the classical case, because of the lack of symmetry between the variables $x_0$ and $x_{j\geq 1}$. Nonetheless, the expected conclusion holds true:

\begin{thm}[\cite{Monod-Py2_arx}]\label{thm:R:rep}
Given a function $F\colon G \to \RR$ of real hyperbolic type, there is an isometric $G$-action on $\HR^\kappa$ for some $\kappa$ and there is $p\in \HR^\kappa$ such that $ F(g) = \cosh d(gp, p)$ holds for all $g\in G$.

Moreover, the orbit $G p$ is total and the $G$-action is continuous if $G$ is endowed with a group topology for which $F$ is continuous.
\end{thm}

In this context, \emph{total} means that it is not contained in a real hyperbolic proper subspace.

\subsection{Two results on real kernels}
Classically, there is a whole calculus preserving functions of positive type: this originates from the fact that we can form sums and tensor products of unitary representations. Neither of these operations are available for hyperbolic representations. It is therefore remarkable that the deformation $t\mapsto \beta^t$ below exists at all, and striking that it is in a sense the \emph{unique} universal deformation available~\cite[Thm.~II]{Monod-Py2_arx}.

The first result presented in this introduction is to \emph{re-prove} the following theorem from~\cite{Monod-Py2_arx}.

\begin{thm}\label{thm:R:power}
If $\beta$ is a kernel of real hyperbolic type, then so is $\beta^t$ for every $0\leq t\leq 1$.
\end{thm}

The proof given in~\cite{Monod-Py2_arx} using~\cite{Monod-Py} relies notably on substantial classical representation theory; specifically, on intertwining integrals for the principal series of $\SO(1,n)$. For this reason, we found it worthwhile to obtain a direct proof as proposed below. Perhaps more crucially, the proof below will also establish a \emph{complex} version of the theorem, which we did not achieve in~\cite{Monod-Py2_arx}.

Theorem~\ref{thm:R:power} implies in particular that there exists a natural one\hyph{}parameter family of exotic representations of $\Isom(\HR^\kappa)$; this is a fascinating interpolation from the identity representation at $t=1$ to the trivial representation at $t=0$. In~\cite{Monod-Py},\cite{Monod-Py2_arx}, this is shown to lead to a complete classification of the continuous representations of $\Isom(\HR^\kappa)$ on real hyperbolic spaces. Let us point out again that the present direct proof of Theorem~\ref{thm:R:power} gives an alternative approach even for $\SO(1, n)$ with $n$ finite for the representations that were defined in~\cite{Monod-Py} using intertwining integrals on the principal series.

\medskip
Another result that we shall re-prove is the fact that for every $\lambda\geq 1$, the kernel $\lambda^d$ is of hyperbolic type when $d$ is the distance function of a simplicial tree. This was established in~\cite{Burger-Iozzi-Monod} by constructing an explicit embedding of simplicial trees into $\HR^\kappa$. Giving a proof directly on the level of kernels has the advantage that it makes no distinction between simplicial trees and more general metric trees such as $\RR$-trees (which would otherwise have to be approximated by simplicial trees). Thus, we prove by a direct argument the following:

\begin{prop}\label{prop:tree}
Let $(X, d)$ be a metric tree. Then the kernel $(x,y)\mapsto \lambda^{d(x,y)}$ is of real hyperbolic type on $X\times X$ for all $\lambda\geq 1$.
\end{prop}

We shall see in Section~\ref{sec:free} that, beyond trees, it is possible to embed ``free products'' of hyperbolic spaces canonically into hyperbolic spaces; see Theorem~\ref{thm:free}.

\medskip
As a curiosity, we point out that one can apply Proposition~\ref{prop:tree} to $\RR$-trees associated to $\GL_2$ by suitable (non-discrete) valuations and deduce the following in view of $\PGL_2(\RR)\cong\Isom(\HR^2)$ and of $\PGL_2(\CC)\cong\Isom(\HR^3)^\circ$.

\begin{thm}\label{thm:wild}
There are irreducible representations of $\Isom(\HR^2)$ and of $\Isom(\HR^3)^\circ$ into $\Isom(\HRI)$ with a countable orbit.
\end{thm}

Having a countable orbit, such representations are in particular very different not only from the continuous representations of $\Isom(\HR^2)$ and of $\Isom(\HR^3)^\circ$ into $\Isom(\HRI)$ considered in~\cite{Delzant-Py},\cite{Monod-Py} but also from the non\hyph{}continuous representations obtained from the latter via wild field automorphisms, see~\cite[p.~533]{Lebesgue07},\allowbreak\cite{Kestelman51}.

\subsection{Complex hyperbolic kernels}
The main contributions of these notes regard \emph{complex hyperbolic spaces} $\HC^\kappa$. There are two closely related reasons why the situation is more subtle than for real spaces:

First, the distance function $d$ does not suffice to capture all the information on tuples of points in $\HC^\kappa$. Indeed, two triangles with pairwise equal side-lengths are generally not congruent in $\HC^\kappa$ even for $\kappa=2$; they can be distinguished by their \emph{Cartan argument} $\cartan\in(-\pi/2, \pi/2)$ recalled in Section~\ref{sec:cartan} below. (Thus even Euclid's \emph{fourth} postulate fails.)

The answer to this objection could seem simple: the distance gives (up to a hyperbolic cosine) the modulus of the relevant kernel, whilst $\cartan$ gives its argument --- thus a complex kernel should suffice. In the context of Hilbertian kernels, this is essentially how the question is easily resolved.

However, the second difficulty is that for complex hyperbolic spaces such a complex kernel (on two variables) cannot be defined in a functorial way on $\HC^\kappa$ and therefore cannot describe group representations. There is a cohomological obstruction related to the projectivization of the linear model. The solution will be to encode this cohomological information in our definition of complex hyperbolic kernels.

\begin{defi}\label{def:KCHT}
A \emph{kernel of complex hyperbolic type} on a set $X$ is a pair $(\beta, \alpha)$ with
\begin{equation*}
\beta\colon X^2 \lra \RR_{\geq 0}, \kern10mm \alpha\colon X^3\lra \left(-\tfrac\pi2, \tfrac\pi2\right)
\end{equation*}
such that $\beta$ takes the constant value~$1$ on the diagonal, $\alpha$ is an alternating $2$-cocycle, and
\begin{equation*}
(x,y) \longmapsto \beta(x, x_0) \beta(x_0, y) - e^{i \alpha(x_0,x,y)} \, \beta(x,y)
\end{equation*}
is of (complex) positive type for all $x_0\in X$.
\end{defi}

Recall here that a function $\alpha\colon X^3\to \RR$ is called a \emph{$2$-cocycle} if its \emph{coboundary} $d \alpha$ vanishes on $X^4$, where $d \alpha$ is defined by
\begin{equation*}
d\alpha (x_0, x_1, x_2, x_3) = \alpha (x_1, x_2, x_3) - \alpha (x_0, x_2, x_3) + \alpha (x_0, x_1, x_3) - \alpha (x_0, x_1, x_2).
\end{equation*}
A cocycle is said to be \emph{alternating} if any permutation of the variables simply replaces $\alpha$ by $\pm \alpha$ according to the sign of the permutation; in particular it vanishes whenever two of its variables coincide. It thus follows from the definition that $\beta$ is symmetric.

\begin{rem}\label{rem:beta:RC}
The trivial map $\alpha=0$ is of course an alternating cocycle. We see that $(\beta, 0)$ is a kernel of complex hyperbolic type if and only if $\beta$ is a kernel of \emph{real} hyperbolic type. Confer also Remark~\ref{rem:proofs:KCHT} below.
\end{rem}

The definition captures the right concept in the following sense.

\begin{prop}\label{prop:CHT:char}
A pair $(\beta, \alpha)$ is a kernel of complex hyperbolic type on the set $X$ if and only if there is a map $f\colon X\to \HC^\kappa$ for some cardinal $\kappa$ such that
\begin{equation*}
\beta(x,y) = \cosh d\big(f(x), f(y)\big) \kern5mm\text{and} \kern5mm \alpha(x,y,z) = \cartan \big(f(x), f(y), f(z)\big)
\end{equation*}
hold for all $x,y, z \in X$.
\end{prop}

\subsection{Complex hyperbolic functions and representations}
Moving on from kernels to groups, we pose the following definition.

\begin{defi}
A \emph{function of complex hyperbolic type} on a group $G$ is a pair $(F, \alpha)$ with $F\colon G\to \RR_{\geq 0}$ such that $(\beta, \alpha)$ is a $G$-invariant kernel of complex hyperbolic type on $G$ when $\beta$ is defined by $\beta(g,h) = F(g\inv h)$.
\end{defi}

Here $G$-invariance is understood with respect to the diagonal actions on $G^2$ and $G^3$. This definition is therefore identical with its real counterpart when $\alpha=0$. In general, $G$-invariant cocycles $\alpha$ are the usual homogeneous cocycles for the second group cohomology of $G$.

The complex hyperbolic result corresponding to Theorem~\ref{thm:R:rep} is as follows.

\begin{thm}\label{thm:C:rep}
Given a function of complex hyperbolic type $(F, \alpha)$ on a group $G$, there is a $G$-action by holomorphic isometries on $\HC^\kappa$ for some $\kappa$ and there is $p\in \HC^\kappa$ such that
\begin{equation*}
F(g) = \cosh d(gp, p) \kern5mm\text{and}\kern5mm \alpha(g_0, g_1, g_2) = \cartan(g_0 p, g_1 p, g_2 p)
\end{equation*}
hold for all $g, g_j \in G$.

Moreover, the orbit $G p$ is total and the $G$-action is continuous if $G$ is endowed with a group topology for which $F$ and $\alpha$ are continuous.
\end{thm}

The entire group of isometries of $\HC^\kappa$ does not quite preserve the Cartan argument since anti-holomorphic isometries will reverse its sign. In general, we say that a cocycle $\alpha$ is \emph{twisted-invariant} under $G$ if every $g\in G$ either preserves $\alpha$ or maps it to $-\alpha$. The proof of Theorem~\ref{thm:C:rep} shows more generally that pairs $(F, \alpha)$ with  $\alpha$ twisted-invariant yield isometric $G$-actions on $\HC^\kappa$. We shall freely use both formulations, i.e.\ representations to $\Isom(\HC^\kappa)^\circ$ and to $\Isom(\HC^\kappa)$.

This is different from the real case, where kernels do not detect the orientation of $\HR^n$, and indeed there is not even a notion of orientation for $\HR^\kappa$ when $\kappa$ is infinite. In the complex case, the complex conjugation is encoded in the sign of $\alpha$ and does indeed survive for $\kappa$ infinite.

\subsection{Complex powers}
As announced above, Theorem~\ref{thm:R:power} is a particular case (namely $\alpha=0$) of the following result.

\begin{thm}\label{thm:C:power}
If $(\beta, \alpha)$ is a kernel of complex hyperbolic type, then so is $(\beta^t, t \alpha)$ for every $0\leq t\leq 1$.
\end{thm}

Using Theorem~\ref{thm:C:rep}, it now follows that $\Isom(\HCI)$ also admits a one\hyph{}parameter family of exotic self\hyph{}representations, which can moreover be shown to be irreducible.

\begin{thm}\label{thm:rep:exist}
For every $0< t\leq 1$ there is an irreducible continuous representation $\ro\colon \Isom(\HCI)\to \Isom(\HCI)$ together with a $\ro$-equivariant map $f\colon\HCI\to\HCI$ satisfying
\begin{align*} 
\cosh d\big(f(x), f(y)\big) &= \big(\cosh d(x,y)\big)^t\kern5mm\text{and}\\
\cartan\big(f(x),f(y),f(z)\big) &= t \,\cartan(x,y,z)
\end{align*}
for all $x,y,z\in\HCI$.
\end{thm}

\begin{rem}
In the proof of irreducibility, we shall see that the representations of Theorem~\ref{thm:rep:exist}, when restricted to $\Isom(\HRI)$, preserve in the target a copy of $\HRI\se \HCI$ and that this restriction coincides with the self\hyph{}representations of $\Isom(\HRI)$ constructed in~\cite{Monod-Py2_arx}.
\end{rem}

Theorem~\ref{thm:C:power} can also be applied to the finite\hyph{}dimensional case $\HC^n$. Just as for $\HR^n$, it turns out that the \emph{irreducible} part of the representations arising from the $t$-power kernel have a slightly different kernel, whose explicit expression is less simple. A more readable, but equivalent, characterisation is provided in that case (as in~\cite{Monod-Py}) by the \emph{translation length} $\ell$ of isometries (recalled in Section~\ref{sec:isom}) and we obtain the following exotic representations of the \emph{ordinary Lie groups} $\Isom(\HC^n)$.

\begin{thm}\label{thm:rep:exist:n}
For every $0< t< 1$ and every integer $n\geq 1$ there is an irreducible continuous representation $\ro\colon \Isom(\HC^n)\to \Isom(\HCI)$ together with an equivariant map $f\colon \HC^n\to\HCI$ satisfying
\begin{equation*} 
\ell\big(\ro(g)\big) = t\, \ell(g) \kern5mm\text{and} \kern5mm \cartan\big(f(x),f(y),f(z)\big) = t \,\cartan(x,y,z)
\end{equation*}
for all $g\in \Isom(\HC^n)$ and all $x,y,z\in\HC^n$.
\end{thm}

Again, these representations contain the representation of $\Isom(\HR^n)$ studied in~\cite{Monod-Py} for $n\geq 2$ and actually provide a new proof of their existence, not using the theory of principal series, see Remark~\ref{rem:contain:n}.

\begin{rem}\label{rem:2:series:state}
The special case of $n=1$ in Theorem~\ref{thm:rep:exist:n} restricts to a family of irreducible continuous representations of $\PSL_2(\RR) \cong \Isom(\HC^1)^\circ$ into $\Isom(\HCI)$ which are \emph{different} from the irreducible continuous representations of $\PSL_2(\RR) \cong \Isom(\HR^2)^\circ$ into $\Isom(\HCI)$ obtained by complexifying the representations studied in~\cite{Delzant-Py},\cite{Monod-Py}, see Remark~\ref{rem:2:series} below for details.
\end{rem}

\subsection{Classification and argument/length rigidity}
Now that we know that $\Isom(\HCI)$ does admit non\hyph{}obvious self\hyph{}representations, the classification question arises. A useful tool here is that any irreducible representation of $\Isom(\HCI)$ or of $\Isom(\HRI)$ to $\Isom(\HCI)$ contains a \emph{unique} orbit isomorphic, as homogeneous space, to $\HCI$ or $\HRI$ respectively (see Section~\ref{sec:fixator}). An equivalent reformulation in more familiar representation\hyph{}theoretical terms is that there is a unique ``spherical'' line fixed by the unitary, repectively orthogonal, subgroup. It follows that more generally non-elementary representations admit a \emph{canonical orbit} isomorphic to $\HCI$ or $\HRI$.

It turns out that the representations constructed in Theorem~\ref{thm:rep:exist} have the property that their (unique) canonical orbit contains the canonical orbit of the subgroup $\Isom(\HRI)$, see Remark~\ref{rem:canonical}. Under this restriction, we obtain a complete classification as follows. 

\begin{thm}\label{thm:rep:classify}
Consider any irreducible self\hyph{}representation $\Isom(\HCI)\to \Isom(\HCI)$ whose canonical orbit contains the canonical orbit of $\Isom(\HRI)$.

Then this self\hyph{}representation is conjugated to one of the representations of Theorem~\ref{thm:rep:exist}.
\end{thm}

\begin{rem}
Notice that no continuity assumption is made; indeed there is an automatic continuity phenomenon exactly as in Theorem~III in~\cite{Monod-Py2_arx}. Moreover, it will be apparent from the proof that Theorem~\ref{thm:rep:classify} remains valid, mutatis mutandis, for the index two subgroups of holomorphic isometries.
\end{rem}

Prehaps all irreducible self\hyph{}representations of $\Isom(\HCI)$ satisfy the above restriction; then Theorem~\ref{thm:rep:classify} would give an unconditional classification. Our additional assumption allows us to combine results from~\cite{Monod-Py} with the argument/length rigidity Theorem~\ref{thm:t:s} below. In general, an unconditional classification could be attempted e.g.\ by revisting the arguments themselves used to prove~\cite{Monod-Py}.

\medskip

We observe a new phenomenon when comparing Theorem~\ref{thm:rep:classify} with the analogous classification obtained in~\cite{Monod-Py2_arx} for the real hyperbolic case. Indeed, a pair $(\beta, \alpha)$ associated to an arbitrary irreducible self\hyph{}representation of $\Isom(\HCI)$ could suggest that a classification would require \emph{two} numerical invariants $t,s$:

A first one would be obtained if we prove as in the real case that the distances are affected by an exponent $t\in (0, 1]$ (in the hyperbolic cosine). Then, a second scalar $s\in \RR$ could appear as follows. The pull-back of the Cartan argument in the target gives a $2$-cocycle for the source group. The latter has one\hyph{}dimensional cohomology and we can hope that the resulting proportionality factor $s$ between cohomology classes finishes to determine the representation.

This is indeed the first outline of our argument. However, it turns out that only $s=\pm t$ is possible in the setting of Theorem~\ref{thm:rep:classify}. This is a manner of (anti-)holomorphy that we establish using the following result.

\begin{thm}\label{thm:t:s}
Let $(\beta, \alpha)$ be the tautological kernel of complex hyperbolic type on the complex geodesic $\HC^1$ and let $0\leq t \leq 1$, $s\in \RR$.

Then $(\beta^t, s \alpha)$ is of complex hyperbolic type only if $s=\pm t$.
\end{thm}

\subsection{A comment on real irreducibility}
Recall that classical \emph{irreducible} unitary representations correspond to functions of positive type that are \emph{pure}, i.e.\ extremal in the convex cone of all functions of positive type, see e.g.~\cite[2.5.2]{Dixmier6996_C}.

In the hyperbolic case, we face again the fact that there is no convex structure. There is nonetheless a clean condition ensuring that a function of hyperbolic type yields an irreducible representation:

\begin{prop}\label{prop:min:irred}
Let $F\colon G\to \RR_{\geq 0}$ be a function of real hyperbolic type on a group $G$. Suppose that the ratio $F/F'$ is unbounded for any other function of real hyperbolic type $F'\leq F$.

Then the representation corresponding to $F$ is irreducible.
\end{prop}

\noindent
The above formulation allows the degenerate (but irreducible) cases where the hyperbolic space has dimension~$0$ or~$1$.

\subsection{A logical conclusion}
In conclusion, we come back to the opening theme of this introduction, namely: the various kernels offer a concrete, first-order way to access representations in functional analysis and geometry.

This apparently vague statement can be understood very precisely in the language of first-order \emph{continuous logic}, or \emph{metric model theory}, exposed in~\cite{BYBHU}. We now sketch this interpretation.

One of the most basic examples of a relational metric structure is provided by class of finite sets endowed with kernels of positive type (say real). Expressions such as~\eqref{eq:def:KPT} are examples of first-order metric relations. This class turns out to be a \fraisse class in the sense of metric model theory. The corresponding ultra\hyph{}homogeneous \emph{\fraisse limit} can be identified with the separable infinite\hyph{}dimensional real Hilbert space. That identification amounts exactly to the GNS construction together with the homogeneity properties of Hilbert space.

\smallskip
In that context, it is unsurprising that the class of finite sets endowed with kernels of real hyperbolic type also constitute a \fraisse class, and that the corresponding \fraisse limit is~$\HRI$. The only point needing a proof is that $\HRI$ is ultra\hyph{}homogeneous. This can be deduced from the ultra\hyph{}homogeneity of Hilbert space. Indeed, $\Isom(\HRI)$ is homogeneous, and the action of a point\hyph{}stabiliser on each corresponding sphere in $\Isom(\HRI)$ is isomorphic to the analogous action for Hilbert spaces. Moreover, the distance functions on hyperbolic spheres are completely determined by the corresponding scalar products in Hilbert space and vice-versa (this is particularly clear in the ball model of $\HRI$). This reduces as claimed the ultra\hyph{}homogeneity of $\HRI$ to that of Hilbert space.

\medskip
\itshape More interesting is the case of $\HCI$.\upshape\ As a metric space, $\HCI$ is not triply homogeneous. In particular, the corresponding \emph{age} of finite subsets with their distance function \emph{is not a \fraisse class}.

However, we can enrich the language to include the Cartan argument, seen as a \emph{ternary predicate} in the sense of continuous logic. This is still a first-order relational language. Now, the class of all finite sets endowed with kernels of complex hyperbolic type is again a \fraisse class and its \fraisse limit can be identified with $\HCI$.

The proof of this statement goes along the lines sketched above for $\HRI$ since again it suffices to establish the ultra\hyph{}homogeneity of $\HCI$, but this time seen as a structure with both $\cosh d$ and $\cartan$. After using transitivity, the point is that we can again reduce to the action of $\UU$ on complex Hilbert spheres, because the additional data of the $\UU$-fixed point $x_0$ allows us to recover the complex scalar product. Indeed, its modulus and argument can be computed from $\cosh d$ and from $\cartan(x_0, \cdot, \cdot)$ respectively.

\subsection*{Acknowledgement}
These notes are motivated by the joint work with Pierre Py~\cite{Monod-Py2_arx}; I am very grateful to Pierre for his comments on a draft version.

\section{Preliminaries}\label{sec:prelim}
\subsection{Hyperbolic spaces}
We set the notation right away in the complex case. Our sesquilinear forms are linear in the second variable. Given a complex Hilbert space $H$, we consider the Hermitian Minkowski form $B$ on $\CC\oplus H$ defined by
\begin{equation*}
B\left( z\oplus u, w \oplus v\right) = \ol z w - \her{u}{v}
\end{equation*}
The \emph{complex hyperbolic space} $\HC^\kappa$ is defined as the projectivization of the set of positive vectors in $\CC\oplus H$, where the cardinal $\kappa$ denotes the Hilbertian dimension of $H$, which may be finite or not. Just as in the finite\hyph{}dimensional case, the real part of $B$ induces a Riemannian metric on $\HC^\kappa$ and a common normalisation choice is that the sectional curvature of $\HC^\kappa$ ranges in $[-4, -1]$. Then the corresponding distance function is given by
\begin{equation}\label{eq:dist}
\cosh d(x,y) = \frac{\left|B(X,Y)\right|}{\big(B(X,X) B(Y,Y)\big)^{\frac12}},
\end{equation}
where $X,Y$ are any vectors representing $x,y$ respectively (noting that this is indeed well-defined projectively).

The \emph{boundary at infinity} $\partial\HC^\kappa$ of $\HC^\kappa$ in the sense of metric geometry can be realised in this model as the space of $B$-isotropic lines in $\CC\oplus H$.

\medskip
The definition of the \emph{real hyperbolic space} $\HR^\kappa$ is exactly as above but with everything over the field $\RR$. In particular we see that there is a copy of $\HR^\kappa$ in $\HC^\kappa$. The sectional curvature restricts to the constant~$-1$ on this copy. At the other extreme, each \emph{complex geodesic} $\HC^1$ in $\HC^\kappa$ is isometric to a rescaled copy $\tfrac12\HR^2$ of the real hyperbolic plane and the sectional curvature attains the value~$-4$ precisely on these complex geodesics.

\medskip
We refer to~\cite{Burger-Iozzi-Monod} for more in-depth background on $\HR^\kappa$; the adjustments for the complex case can be taken from the finite\hyph{}dimensional monograph~\cite{Goldman}. (We warn the reader that the former reference uses the East Coast signature convention and that the latter reference chooses a different normalisation of the metric.)

In particular, we note that all hyperbolic spaces are complete by construction in the above Minkowski model; a more intrinsic (model-free) approach to quadratic spaces is discussed in~\cite[\S2-3]{Burger-Iozzi-Monod}. This approach also makes the following fact transparent: given a non-empty set $X\se \HC^\kappa$, there is a unique smallest hyperbolic subspace of $\HC^\kappa$ containing $X$; the latter can be viewed as the hyperbolic space obtained from the form $B$ restricted to closed $\CC$-linear space spanned by any lift of $X$ to $\CC\oplus H$. The subset $X$ is called \emph{total} if that smallest hyperbolic subspace is $\HC^\kappa$ itself. Again, the corresponding facts and terminology are available over~$\RR$.

\subsection{Isometries}\label{sec:isom}
We write $\UU(B)$ or $\UU(1, \kappa)$ for the group of (necessarily continuous) invertible operators preserving $B$; we thus get an isometric representation of the projectivised group $\PU(B)$ on $\HC^\kappa$. Its image is precisely the group $\Isom(\HC^\kappa)^\circ$ of holomorphic isometries of $\HC^\kappa$, which is the neutral component of the isometry group $\Isom(\HC^\kappa)$. The entire group $\Isom(\HC^\kappa)$ also contains the anti-holomorphic isometries.

In the real case, the quotient $\PO(B)$ of $\OO(B)$ (also denoted $\PO(1, \kappa)$ and $\OO(1, \kappa)$) identifies with the group $\Isom(\HR^\kappa)$. When $\kappa$ is finite, $\Isom(\HR^\kappa)^\circ$ is the index two subgroup of orientation-preserving isometries, whilst $\Isom(\HRI)$ is connected. We clarify at this point that our isometry groups are always endowed with the topology of pointwise convergence.

\medskip
The space $\HC^\kappa$ is homogeneous under the action of $\Isom(\HC^\kappa)^\circ$ and the explicit Minkowski model above shows that point\hyph{}stabilisers in $\Isom(\HC^\kappa)^\circ\cong \PU(B)$ are isomorphic to the unitary group $\UU(H)$ of the Hilbert space $H$. The stabiliser in $\Isom(\HC^\kappa)$ is the index two over\hyph{}group $\tUU(H)$ of $\UU(H)$ containing the complex conjugation.

In the real case, point\hyph{}stabilisers in $\Isom(\HR^\kappa)$ identify with $\OO(H)$ and for $\kappa$ finite, stabilisers in $\Isom(\HR^\kappa)^\circ$ are $\SO(H)$. A fundamental fact is that our isometric representations will map point stabilisers to point stabilisers:

\begin{lem}\label{lem:fixed:pt}
For any continuous representation of $\Isom(\HC^\kappa)$ by isometries of a complete \cat0 space $X$, the stabiliser of a point in $\HC^\kappa$ fixes a point in $X$. The corresponding statement also holds for $\HR^\kappa$.
\end{lem}

\begin{proof}
By the Cartan fixed-point theorem~\cite[II.2.8]{Bridson-Haefliger}, it suffices to know that this stabiliser has a bounded orbit in $X$. If $\kappa$ is finite, this follows from compactness. If $\kappa$ is infinite, this is a result of Ricard--Rosendal even without assuming continuity~\cite{Ricard-Rosendal}; this holds for both $\UU$ and $\OO$, see p.~190 in~\cite{Ricard-Rosendal}.
\end{proof}

Recall that the \emph{translation length} $\ell(g)$ of an isometry $g$ of an arbitrary metric space $X$ is $\ell(g)=\inf_{x\in X} d(g x, x)$. In the \cat0 case, this quantity can be recovered as $\lim_{n\to\infty} \tfrac1n d(g^n x, x)$ for any choice of $x\in X$, see~\cite[6.6(2)]{Ballmann-Gromov-Schroeder}. In particular, it follows that for any function of complex hyperbolic type $(F, \alpha)$ on a group $G$ we have
\begin{equation}\label{eq:length:F}
\ell(g) =   \log \lim_{n\to\infty} F(g^n)^{\frac1n} \kern10mm(\forall\,g\in G)
\end{equation}
where $g$ is viewed as an isometry of $\HC^\kappa$ according to Theorem~\ref{thm:C:rep}. The case $\alpha=0$ yields the corresponding statement for functions of real hyperbolic type, already recorded in~\cite[\S4.B]{Monod-Py2_arx}.

\subsection{Non-elementary and irreducible representations}\label{sec:non-el}
In accordance with the terminology used more generally in \cat{-1} spaces, an isometric action of a group $G$ on $\HC^\kappa$ is called \emph{elementary} if $G$ preserves a finite set in $\HC^\kappa \cup \partial \HC^\kappa$. This happens if an only if $G$ either fixes a point in $\HC^\kappa \cup \partial \HC^\kappa$ or preserves a geodesic line in $\HC^\kappa$.

If the $G$-action is non\hyph{}elementary, then there exists a unique minimal $G$-invariant hyperbolic subspace in $\HC^\kappa$ (which might be a complex geodesic). This is proved as in Proposition~4.3 of~\cite{Burger-Iozzi-Monod}. We call the action \emph{minimal} if this subspace is $\HC^\kappa$ itself; equivalently, if every $G$-orbit is total.

Recalling that the isometric $G$-action on $\HC^\kappa$ corresponds to a projective representation on $\CC\oplus H$ ranging in $\PU(B)$, we have the following equivalence when $\kappa\geq 2$:\itshape\ the $G$-action is non\hyph{}elementary and minimal if and only if the projective representation is irreducible.\upshape\ This holds mutatis mutandis in the simpler case of $\HR^\kappa$, see~\cite{Burger-Iozzi-Monod}.

For perfect groups, a basic non\hyph{}elementarity criterion is available (this idea was used in~\cite{Monod-Py2_arx} in the real case).

\begin{lem}\label{lem:non-elem}
Let $(F,\alpha)$ be a function of complex hyperbolic type on a group $G$ and suppose that $\lim_{n\to\infty} F(g^n)^{\frac1n} >1$ for some $g\in G$. If $G$ is perfect, then the corresponding $G$-action on $\HC^\kappa$ is non\hyph{}elementary.

When $G$ is a topological group and $(F,\alpha)$ is continuous, the same conclusion holds under the weaker assumption that $G$ is topologically perfect.

Finally, the conclusion also holds if $g$ belongs to some subgroup of $G$ which is perfect, respectively topologically perfect.
\end{lem}

\begin{proof}
There is no loss of generality in assuming that $G$ itself is the (topologically) perfect subgroup. The assumption on $g$ implies that $g$ has positive translation length because of the formula~\eqref{eq:length:F}. In particular $g$ cannot fix a point in $\HC^\kappa$. Supposing for a contradiction that the $G$-action is elementary, it follows that $G$ fixes either a point at infinity or a pair of points at infinity. The (topological) perfection of $G$ implies that the former is the case and that moreover $G$ preserves all horospheres based at that point. In particular, $g$ is parabolic; in \cat{-1} geometry, this contradicts the positivity of the translation length.
\end{proof}

\subsection{The Cartan argument}\label{sec:cartan}
We first recall the intrinsic geometric definition. Given three points $x,y,z$ in a complex hyperbolic space $\HC^\kappa$, the \emph{Cartan argument} can be defined by
\begin{equation*}
\cartan(x,y,z) = \int_{\Delta(x,y,z)} \omega\kern3mm \in \left( -\tfrac\pi2, \tfrac\pi2 \right)
\end{equation*}
where $\omega$ is the K\"ahler form and $\Delta(x,y,z)$ is any simplex with geodesic sides on the vertices $x,y,z$. A number of points have to be clarified to make this a well-posed definition. First, there is nothing special if the cardinal $\kappa$ is infinite since any triple of points is contained in a copy of $\HC^2$ in $\HC^\kappa$. Secondly, the integral on the right hand side does not depend on the particular choice of the simplex since $\omega$ is closed. Finally, we normalise $\omega$ in such a way that the range of the Cartan argument is the whole open interval $(-\pi/2, \pi/2)$.

\begin{rem}\label{rem:omega:area}
Since our normalisation of the metric is such that complex geodesics have constant sectional curvature~$-4$, it follows that the K\"ahler form as normalised above restricts to \emph{twice} the area form on complex geodesics.
\end{rem}

An equivalent definition by means of the Hermitian triple product, closer to the original definition by Cartan~\cite[p.~167]{ECartan32}, is as follows. Suppose that $x,y,z$ are represented respectively by positive vectors $X, Y, Z$ for a Hermitian form $B$ of index one. The triple product
\begin{equation}\label{eq:triple}
\left\langle X, Y, Z\right\rangle = B(X, Y) \, B(Y, Z)\, B(Z, X)
\end{equation}
scales by a positive scalar whenever any of the vectors is multiplied by a complex scalar; therefore its argument is represented by an element of $\SSS^1$ that depends only on $x,y,z$. Moreover, this element of $\SSS^1$ has positive real part: this follows algebraically from the fact that $X,Y,Z$ are positive vectors. Therefore we can represent this argument by an element of $(-\pi/2, \pi/2)$, and this coincides with $\cartan(x,y,z)$.

Since this algebraic definition is the one that we shall use, we recall why it satisfies the following basic properties (which are clear for the definition in terms of the K\"ahler form).

\begin{lem}\label{lem:cartan}
The Cartan argument is an alternating $2$-cocycle.
\end{lem}

\begin{proof}
Since $B$ is Hermitian, $\cartan$ is alternating. Moreover, the element of $\SSS^1$ defined by $\left\langle X, Y, Z\right\rangle$ is a (multiplicative) cocycle because it is actually a coboundary: the coboundary of the image of $B$ in $\SSS^1$. It follows that $\cartan$ satisfies the cocycle equation modulo $2\pi \ZZ$ in $\RR$. But this equation has only four terms, all less than~$\pi/2$ is absolute value. Therefore the cocycle relation is already satisfied in~$\RR$.
\end{proof}

\begin{rem}\label{rem:cohomologous}
There is no mystery as to why the two definitions of $\cartan$ coincide. Indeed, the fact that any pair of points in $\HC^n$ can be interchanged by an element of $G=\Isom(\HC^n)^\circ$ shows that there is no non-zero alternating $G$-invariant function on $\HC^n\times \HC^n$ . This implies that every class in degree two continuous real cohomology of $G$ is represented by a \emph{unique} alternating $2$-cocycle on $\HC^n$.

This shows that we only need to make the right normalisation choice for $\cartan$ since the cohomology $\hc^2(G, \RR)$ is well-known to be one\hyph{}dimensional (it suffices here to argue with $n$ finite).

The same argument shows that if $X$ is a Riemannian symmetric space and $G=\Isom(X)$, then every class in the second continuous cohomology of $G$ with real coefficients is represented by a \emph{unique} alternating $G$-invariant cocycle on $X^3$.
\end{rem}

The Cartan argument detects totally real subspaces as follows. (For a similar statement valid also at infinity, see~\cite[2.1]{Burger-IozziPSUn1_pb}.)

\begin{lem}\label{lem:tot:real}
Let $X\se\HC^\kappa$ be any subset. The Cartan argument vanishes on all triples in $X$ if an only if $X$ is contained in a totally real subspace of $\HC^\kappa$.
\end{lem}

\begin{proof}
The vanishing of $\cartan$ implies that $\cosh d$ is a kernel of real hyperbolic type on $X$ (Remark~\ref{rem:beta:RC}); therefore, by Proposition~\ref{propdef:KHT}, there is for some $\kappa'$ an isometric copy of $\HR^{\kappa'}$ in $\HC^\kappa$ containing $X$. The fact that these isometric copies coincide with the totally real subspaces is proved as in the finite\hyph{}dimensional case, see e.g.~\cite[II.10.16]{Bridson-Haefliger}. The converse follows from the definition of the Cartan argument.
\end{proof}

\section{Kernel yoga and functional pilates}
\begin{flushright}
\begin{minipage}[t]{0.7\linewidth}\itshape\small
\begin{flushright}
Comme d'habitude, nous commen\c{c}ons par une sorte\\
d'axiomatisation des propri\'et\'es que nous avons en vue.
\end{flushright}
\begin{flushright}
\upshape\small
Grothendieck~\cite{Grothendieck53} page 158.
\end{flushright}
\end{minipage}
\end{flushright}
\subsection{The characterisation of kernels}\label{sec:yoga}
The main step in the proof of Proposition~\ref{prop:CHT:char} is the following.

\begin{prop}\label{prop:KCHT}
Let $X$ be a set with a kernel of complex hyperbolic type $(\beta, \alpha)$.

Then there is a map $f\colon X\to \HC^\kappa$ with total image in a complex hyperbolic space (for some cardinal $\kappa$) such that
\begin{equation*}
\beta(x,y) = \cosh d\big(f(x), f(y)\big) \kern5mm\text{and} \kern5mm \alpha(x, y, z) = \cartan \big(f(x), f(y), f(z)\big)
\end{equation*}
hold for all $x,y,z\in X$.
\end{prop}

\begin{proof}
Choose some $x_0\in X$. The GNS construction for complex kernels of positive type provides a complex Hilbert space $H$ and a map $h\colon X\to H$ such that
\begin{equation}\label{eq:GNS}
\her{h(x)}{h(y)} =  \beta(x, x_0) \beta(x_0, y) - e^{i \alpha(x_0,x,y)} \, \beta(x,y)
\end{equation}
holds for all $x,y\in X$. Now endow the space $\CC \oplus H$ with the Minkowski Hermitian form $B$ and define $f\colon X\to \CC\oplus H$ by $f(x) = \beta(x, x_0)\oplus h(x)$. Then for all $x,y\in X$ we have
\begin{equation}\label{eq:B:beta}
B\big(f(x), f(y)\big) = e^{i \alpha(x_0,x,y)} \, \beta(x,y).
\end{equation}
This shows that $\beta(x,y)= \left|B(f(x), f(y)) \right| = \cosh d(f(x), f(y))$ for all $x,y$.

Next, consider any points $x,y,z\in X$. Computing the Hermitian triple product $\left\langle f(x), f(y), f(z) \right\rangle$ using~\eqref{eq:B:beta}, we find
$$\cartan \big(f(x), f(y), f(z)\big) = \alpha(x_0, x, y) +  \alpha(x_0, y, z) +  \alpha(x_0, z, x).$$
Using that $\alpha$ is an alternating cocycle, we conclude $\cartan(f(x), f(y), f(z)) =  \alpha(x, y, z)$ as desired.

Finally, the fact that $f(X)$ is total in $\CC \oplus H$ follows from the facts that $h(X)$ is total in $H$ and that $f(x_0)$ is $1\oplus 0$ since~\eqref{eq:GNS} forces $\|h(x_0)\|=0$.
\end{proof}

\begin{proof}[End of the proof of Proposition~\ref{prop:CHT:char}]
In view of Proposition~\ref{prop:KCHT}, it remains only to verify that conversely the pair $(\cosh d, \cartan)$ is a kernel of complex hyperbolic type on $\HC^\kappa$ itself. Given Lemma~\ref{lem:cartan}, it suffices to prove that the kernel
\begin{equation}\label{eq:proof:KCHT}
\Phi(x,y) = \cosh d(x, x_0) \cosh d(x_0, y) - e^{i \cartan(x_0, x, y)} \, \cosh d(x,y)
\end{equation}
is of positive type on $\HC^\kappa$ for all $x_0$. In view of the transitivity of $\UU(B)$ on $\HC^\kappa$, we can suppose that $x_0$ corresponds to the vector $1\oplus 0$. We define a (non\hyph{}equivariant!) lift $\HC^\kappa\to\CC\oplus H$ mapping $x$ to $X$ as follows. Since $x$ is a positive line, we can choose $X$ such that $B(X, X)=1$ and such that moreover its first coordinate is in $\RR_{\geq 0}$. It follows that $X=\sqrt{1+\|\tilde x\|^2} \oplus \tilde x$ for some $\tilde x\in H$ and that $\sqrt{1+\|\tilde x\|^2}=\cosh d(x_0, x)$. In particular $\widetilde{x_0}=0$. Moreover, the definition~\eqref{eq:dist} shows that $|B(X, Y)| = \cosh d(x, y)$. Now the argument $\cartan(x_0, x, y)$ of the triple product $\left\langle X_0, X, Y\right\rangle$ as defined in~\eqref{eq:triple} is just the argument of $B(X, Y)$. Summarising, the kernel~\eqref{eq:proof:KCHT} reads
\begin{equation}\label{eq:proof:KCHT:Phi}
\Phi(x, y) = \sqrt{1+\|\tilde x\|^2} \, \sqrt{1+\|\tilde{\smash{y}\vphantom{x}}\|^2} - B(X,Y) = \her{\tilde x}{\tilde y}
\end{equation}
In conclusion, $\Phi$ is indeed of positive type.
\end{proof}

\begin{rem}\label{rem:proof:KCHT:Phi}
The explicit lift $x\mapsto X$ used above, which gives~\eqref{eq:proof:KCHT:Phi}, corresponds to the hyperboloid model in the real case. Although that model is \emph{not equivariant} in the complex case, it can be convenient for calculations. For instance, we used $B(X,Y) = e^{i \cartan(x_0, x, y)} \, \cosh d(x, y)$.
\end{rem}

\begin{rem}\label{rem:proofs:KCHT}
Since the proof of Proposition~\ref{prop:KCHT} starts with an arbitrary choice for $x_0\in X$, we see that we can replace ``for all $x_0$'' by ``there exists $x_0$'' in Definition~\ref{def:KCHT}, just as in the real case. We also see that the case $\alpha=0$ of Proposition~\ref{prop:CHT:char} gives Proposition~\ref{propdef:KHT}.
\end{rem}

\subsection{The equivariant case}
The arguments of the previous section are not functorial because of the asymmetry in the role of $x_0$ and of the other points. This will be mended by the following.

\begin{thm}\label{thm:KCHT-char}
Let $X$ be a set with a kernel of complex hyperbolic type $(\beta, \alpha)$. 

The map $f\colon X\to \HC^\kappa$ of Proposition~\ref{prop:KCHT} is unique up to a unique holomorphic isometry. Therefore, there is a canonical representation of the group of bijections of $X$ preserving $\beta$ and $\alpha$ to $\Isom(\HC^\kappa)^\circ$ for which $f$ is equivariant.
\end{thm}

This results immediately implies Theorem~\ref{thm:C:rep} from the introduction by setting $X=G$. The point $p$ of that statement is then $f(e)$ and the continuity assertion is a standard consequence of the uniform equicontinuity of isometries. The proof below yields a representation to $\Isom(\HC^\kappa)$ if we consider more generally bijections for which $\alpha$ is only twisted-invariant.

\begin{proof}[Proof of Theorem~\ref{thm:KCHT-char}]
Denote by $G$ the group of bijections of $X$ preserving $\beta$ and $\alpha$. We choose an arbitrary point $z\in X$ and define a Hermitian form $B_z$ on the free complex vector space $\CC[X]$ on $X$ by setting
\begin{equation*}
B_z(\fhi, \psi) = \sum_{x,y\in X} e^{i \alpha(z,x,y)} \, \beta(x,y) \, \ol{\fhi(x)} \psi(y)\kern10mm(\fhi, \psi\in \CC[X]).
\end{equation*}
We caution that a priori the set of $B_z$-positive vectors depends on the choice of $z$, although one can check, using the cocycle relation for $\alpha$, that $\fhi$ is $B_z$-positive if and only if the vector
\begin{equation*}
x\longmapsto e^{i \alpha(z',z,x)}\, \fhi(x)
\end{equation*}
is $B_{z'}$-positive. This suggests to define a representation of $G$ on $\CC[X]$ as follows:
\begin{equation*}
(g. \fhi)(x) = e^{i \alpha(z, g z, x)} \,\fhi( g\inv x).
\end{equation*}
Indeed, a direct computation using the cocycle identity for $\alpha$ and the $G$-invariance of $(\beta, \alpha)$ shows $B_z(g. \fhi, g. \psi) = B_z(\fhi, \psi)$. Now $\fhi \mapsto g.\fhi$ is a linear map, but does not quite define a linear representation of $G$. In fact, a computation with the cocycle relation for $\alpha$ shows that we have the following identity for $g,h\in G$\,:
\begin{equation*}
\big(g.(h.\fhi)\big)(x) = e^{i \alpha(z, g z, g h z)} \, \big((gh).\fhi\big)(x).
\end{equation*}
In other words, we have a well-defined projective representation of $G$ preserving $B_z$; notice that the associated multiplier is precisely given by $e^{i \alpha}$.

At this point, the proof is completed exactly as in the real case, see Theorem~3.4 in~\cite{Monod-Py2_arx}. We refer the reader to that proof, but still sketch it briefly. Since $G$ preserves $B_z$, the projective representation descends to the quotient of $\CC[X]$ by the kernel of this Hermitian form. Moreover, the definition of $B_z$ shows that there is an isomorphism between this quotient and a dense subspace of the Minkowski space constructed for Proposition~\ref{prop:KCHT} which is compatible with the maps from $X$ to each of these two spaces. It follows that the completion of the quotient of $\CC[X]$ can be identified with that Minkowski space and all statements follow.
\end{proof}

\section{Powers of kernels and their representations}
\subsection{Powers of kernels}
Let $(\beta, \alpha)$ be a kernel of complex hyperbolic type. We  now proceed to establish Theorem~\ref{thm:C:power}, namely: $(\beta^t, t \alpha)$ is still a kernel of complex hyperbolic type whenever $0\leq t \leq 1$.

In the proof below, we shall take complex $t$-powers; therefore we specify that we define them by principal value. In fact, the proof will only involve $t$-powers of complex numbers with positive real parts due to the fact that $\alpha$ ranges in $(-\pi/2, \pi/2)$.
\begin{flushright}
\begin{minipage}[t]{0.7\linewidth}\itshape\small
\begin{flushright}
in other words, power is a word the meaning of which\\
we do not understand.
\end{flushright}
\begin{flushright}
\upshape\small
Leo Tolstoy, \emph{War and Peace} (tr. by A.\&L. Maude),\\
Chap.~V of the Second Epilogue.
\end{flushright}
\end{minipage}
\end{flushright}
\begin{proof}[Proof of Theorem~\ref{thm:C:power}]
We assume throughout that $0<t<1$ holds, since the cases $t=0,1$ are trivial. Fix $n\in \NN$ and choose points $x_0, x_1, \ldots, x_n$ in $X$. Define the $n$-by-$n$ Hermitian matrix $C$ by
\begin{equation*}
C_{j,k}= e^{i \alpha(x_0, x_j, x_k)}\, \beta (x_j, x_k) \kern10mm(j,k\geq 1)
\end{equation*}
and the vector $b\in\CC^n$ by $b_j=\beta(x_0, x_j)$. Recall that $b_j$ is real and positive (in fact~$\geq 1$). We know that the matrix
\begin{equation}\label{eq:power:hyp}
M= b b\ctran - C
\end{equation}
is positive semi-definite (here $b\ctran$ is just the transpose). Our goal is to prove that the difference of Hadamard powers
\begin{equation}\label{eq:power:conc}
\left( b_j^t \, b_k^t - C_{j,k}^t \right)_{j,k \geq 1}
\end{equation}
is positive semi-definite as well. Since all $b_j$ are positive, we can define a diagonal matrix $D=D\ctran$ by $D_{j,j} = 1/b_j$. Thus $D b = \one$. We define $M'$ by
\begin{equation*}
M' = D M D\ctran = \one \one\ctran  - D C D\ctran,
\end{equation*}
where we used~\eqref{eq:power:hyp}. By construction, $M'$ is positive semi-definite. Moreover, we have
\begin{equation*}
M'_{j,j} = 1 -  b_j^{-1}  \,C_{j,j} \, b_j^{-1} =  1 - b_j^{-2} < 1
\end{equation*}
for all $j\geq 1$. It follows $|M'_{j,k}|<1$ for all $j,k\geq 1$ since $M'$  is positive semi-definite. In particular we can apply to each $M'_{j,k}$ the power series
\begin{equation}\label{eq:q}
q(z) = t\, \sum_{m=1}^\infty \frac{(1-t) (2-t) \cdots (m-1-t)}{m!}\, z^m
\end{equation}
because it has radius of convergence~$1$. We define thus $M''_{j,k} = q(M'_{j,k})$. In fact, $q$ is essentially a generalised binomial series:
\begin{equation*}
q(z) = 1-  \sum_{m=0}^\infty
\begin{pmatrix}
m -1 -t\\
m
\end{pmatrix}
\,z^m
\end{equation*}
and thus basic binomial identities imply
\begin{equation}\label{eq:q:t}
q(z) = 1- (1-z)^t \kern5mm (\forall\,  |z|<1).
\end{equation}
We now use the constraint $0<t<1$. It implies that the expression~\eqref{eq:q} is a limit of \emph{positive} combinations of powers; hence, Schur's product theorem ensures that $M''$ is also positive semi-definite. Thus finally the matrix $N$ defined by
\begin{equation*}
N_{j,k} = b_j^t \, M''_{j,k} \, b_k^t
\end{equation*}
is positive semi-definite. Using the determination~\eqref{eq:q:t} for $q$ and $b_j, b_k>0$, we see that
\begin{equation*}
N_{j,k} = b_j^t \, b_k^t \left(1- \left( 1- M'_{j,k} \right)^t \right) = b_j^t \, b_k^t -  \left( b_j \, b_k - b_j \, M'_{j,k} \,b_k \right)^t
\end{equation*}
and hence
\begin{equation*}
N_{j,k} = b_j^t \, b_k^t -  \left( b_j \, b_k - M_{j,k} \right)^t = b_j^t \, b_k^t  - C_{j,k}^t.
\end{equation*}
In conclusion, we proved indeed that~\eqref{eq:power:conc} is positive semi-definite. 
\end{proof}

The above proof can be interpreted as a \emph{non-equivariant} detour through the ball model.

\subsection{Representations of $\Isom(\HCI)$ associated to powers}
Consider the ``tautological'' kernel $(\beta, \alpha)$ on $\Isom(\HCI)$ given by $\cosh d$ and $\cartan$. According to Theorem~\ref{thm:C:rep}, the power kernel $(\beta^t, t \alpha)$ provided by Theorem~\ref{thm:C:power} gives rise to a continuous self\hyph{}representation $\ro$ of $\Isom(\HCI)$ such that
\begin{align}
\cosh d\big(f(x), f(y)\big) &= \big(\cosh d(x,y)\big)^t\kern5mm\text{and}\label{eq:power:map:d}\\
\cartan\big(f(x),f(y),f(z)\big) &= t \,\cartan(x,y,z)\label{eq:power:map:c}
\end{align}
hold for all $x,y,z\in\HCI$, where $f\colon\HCI\to\HCI$ is given by the $\ro$-orbital map associated to the point $p\in\HCI$ of Theorem~\ref{thm:C:rep}. What remains to be established for Theorem~\ref{thm:rep:exist} is that this representation $\ro$ is irreducible.

Given the relation~\eqref{eq:power:map:d}, it follows from Lem\-ma~\ref{lem:non-elem} (applied to any element $g$ that is hyperbolic in the source group) that the representation $\ro$ is non\hyph{}elementary and hence has a unique irreducible part $\HC^\kappa\se \HC^\infty$. We need to prove $\HC^\kappa= \HC^\infty$, or equivalently $p\in \HC^\kappa$.

\smallskip

Choose a point $p'\in \HC^\kappa$ fixed by some point stabilizer in the source group $\Isom(\HCI)$; for instance, take for $p'$ the image of $p$ under the nearest-point projection to $\HC^\kappa$. The associated orbital map is a $\ro$-equivariant map $f'\colon\HCI\to\HC^\kappa$. We furthermore consider a copy of $\Isom(\HRI)$ in $\Isom(\HCI)$ induced by the inclusion of $\OO(1, \infty)$ in $\UU(1, \infty)$ together with the corresponding $\HRI\se\HCI$. We claim that the pull-back via $f'$ of the Cartan argument on $\HC^\kappa$ vanishes on $\HRI$; more generally, we have the following.

\begin{lem}\label{lem:pull-back:0}
Let $2\leq n\leq \infty$ and let $\gamma$ be any $\Isom(\HC^n)^\circ$-invariant alternating cocycle on $\HC^n$. Then $\gamma$ vanishes on triples of points in $\HR^n$.
\end{lem}

\begin{proof}
We can assume $n$ finite upon possibly factoring through $\Isom(\HC^{n'})^\circ$ for some $n'$ finite with $2\leq n'\leq n$. Now recall that the restriction map in continuous cohomology
\begin{equation*}
\hc^2(\Isom(\HC^n)^\circ, \RR) \lra \hc^2(\Isom(\HR^n)^\circ, \RR)
\end{equation*}
vanishes. This follows from the fact that the left hand side is generated by the Cartan argument of $\HC^n$, which vanishes on $\Isom(\HR^n)^\circ$ (and actually the right hand side vanishes altogether unless $n=2$). It remains to upgrade the vanishing of the \emph{class} of $\gamma$ to the actual vanishing of $\gamma$; this follows as in Remark~\ref{rem:cohomologous}, but this time using the metric double transitivity of the real hyperbolic space.
\end{proof}

For later reference, we observe that the proof given here for Lemma~\ref{lem:pull-back:0} also establishes the following simpler variant.

\begin{lem}\label{lem:pull-back:1}
Let $3\leq n\leq \infty$ and let $\gamma$ be any $\Isom(\HR^n)^\circ$-invariant alternating cocycle on $\HR^n$. Then $\gamma$ vanishes on triples of points in $\HR^n$.\qed
\end{lem}

Combining the above pull-back claim with Lemma~\ref{lem:tot:real}, we deduce that $f'(\HRI)$ is contained in a totally real subspace of $\HC^\kappa$, which we denote by $\HR^{\kappa'}$. Upon taking the smallest such subspace, we can assume that $\HR^{\kappa'}$ is invariant under $\ro(\Isom(\HRI))$. This real representation of $\Isom(\HRI)$ is still non\hyph{}elementary and hence it contains a unique $\RR$-irreducible component $\HR^{\kappa''}\se \HR^{\kappa'}$.

We now apply to $\HR^{\kappa''}$ the classification of irreducible self\hyph{}representations of $\Isom(\HRI)$ established in~\cite{Monod-Py2_arx}. The translation lengths under $\ro$ of elements $g\in \Isom(\HRI)$ are unaffected whether we compute them in $\HCI$ or in $\HR^{\kappa''}$; this is apparent in the expression~\eqref{eq:length:F}. Thus we have $\ell(\ro(g)) = t\, \ell(g)$ in view of~\eqref{eq:power:map:d}. Therefore, the classification (Theorem~II in~\cite{Monod-Py2_arx}) shows that we must have
\begin{equation}\label{eq:f''}
\cosh d\big(f''(x), f''(y)\big) = \big(\cosh d(x,y)\big)^t
\end{equation}
for all $x,y$ in $\HRI\se\HCI$, where $f''$ is the unique $\ro(\Isom(\HRI))$-equivariant map $f''\colon\HRI\to\HR^{\kappa''}$. The uniqueness of this map follows from Remark~5.6 in~\cite{Monod-Py2_arx}. We note that $f''$ can be given as the $\ro(\Isom(\HRI))$-orbital map of the point $p''=\pi(p)$, where $\pi\colon \HCI\to\HR^{\kappa''}$ is the nearest-point projection. Equivalently, $f''=\pi\circ f$ holds on $\HRI$.

Since the values in~\eqref{eq:power:map:d} coincide with those in~\eqref{eq:f''}, we conclude that $\pi$ is the identity, because otherwise the sandwich lemma (see~\cite[II.2.12]{Bridson-Haefliger}) would force
\begin{equation*}
d\big(f''(x), f''(y)\big) < d\big(f(x), f(y)\big).
\end{equation*}
In particular, $p=p''$ lies in $\HR^{\kappa''}$ and thus a fortiori in $\HC^\kappa$, as was to be shown. This finishes the proof of Theorem~\ref{thm:rep:exist}.\qed


\subsection{Representations of $\Isom(\HC^n)$ for $n$ finite}
The only, but significant, difference with the proof of Theorem~\ref{thm:rep:exist} is as follows. We still have a representation $\ro_0\colon \Isom(\HC^n)\to \Isom(\HCI)$ and an equivariant map $f_0\colon \HC^n\to\HCI$ satisfying the same conditions~\eqref{eq:power:map:d} and~\eqref{eq:power:map:c} as above; again, $f_0$ is the orbital map of a point $p_0\in\HCI$ provided by Theorem~\ref{thm:C:rep}. \emph{However}, the representation is not irreducible and we will need to replace the map $f_0$ with another map $f$, given by another point $p$.

As already mentioned above, Lemma~\ref{lem:non-elem} implies that the representation is non\hyph{}elementary and hence has a unique irreducible part. Let $\HC^\kappa\se\HCI$ be the corresponding minimal invariant complex hyperbolic subspace. (We have $\kappa=\infty$, e.g.\ by the Karpelevich--Mostow theorem, but we keep the notation $\HC^\kappa$ to distinguish it from the larger ambient space.) We define $\ro$ to be the co-restriction of $\ro_0$ to $\HC^\kappa$ and $f=\pi\circ f_0$, where $\pi\colon\HCI\to \HC^\kappa$ is the nearest-point projection, which is equivariant. Thus $f$ can also be described as the orbital map associated to $p=\pi(p_0)$.

We still have $\ell(\ro(g)) = \ell(\ro_0(g))$ for all $g$ in $\Isom(\HC^n)$ using~\eqref{eq:length:F}; therefore, we can deduce $\ell(\ro(g)) = t\, \ell(g)$ from~\eqref{eq:power:map:d}. It only remains to show that the identity
\begin{equation}\label{eq:coc:equal}
\cartan\big(f(x),f(y),f(z)\big) = t \,\cartan(x,y,z)
\end{equation}
holds. We know already that~\eqref{eq:coc:equal} holds for $f_0$ instead of $f$. It is a general fact in the cohomology of groups that the cocycles $\cartan(f(x),f(y),f(z))$ and $\cartan(f_0(x),f_0(y),f_0(z))$ are cohomologous (as $\Isom(\HC^n)$-cocycles) because the maps $f$ and $f_0$ are conjugated to each other by an element of $\Isom(\HCI)$, namely any element sending $p$ to $p_0$. In the present case, this implies that the actual cocycles coincide as explained in Remark~\ref{rem:cohomologous}.

The proof of Theorem~\ref{thm:rep:exist:n} is complete; we observe that there is nothing special happening to this argument in the special case $n=1$ since we have not used the restriction to $\Isom(\HR^n)$, in contrast to the proof of Theorem~\ref{thm:rep:exist}.\qed

\begin{rem}
A more geometric argument to deduce~\eqref{eq:coc:equal} from the corresponding identity with $f_0$ is as follows. The maps $f$ and $f_0$ remain at constant distance of each other, namely $d(p, p_0)$. Therefore the equality can be deduced from the fact that the cocycle $\cartan$ is determined by its values at infinity, which is explained in~\cite{Burger-IozziPSUn1_pb}.
\end{rem}

\begin{rem}\label{rem:contain}
In the proof of Theorem~\ref{thm:rep:exist}, we noted that our self\hyph{}representations of $\Isom(\HCI)$, when restricted to $\Isom(\HRI)$, contain the real self\hyph{}representations of $\Isom(\HRI)$ constructed in \cite{Monod-Py2_arx}. A corresponding statement holds for the above representations of $\Isom(\HC^n)$; we shall justify it later in Remark~\ref{rem:contain:n} in order to benefit from some general observations.
\end{rem}

\begin{rem}\label{rem:2:series}
We now justify Remark~\ref{rem:2:series:state}. On the one hand, Theorem~\ref{thm:rep:exist:n} applies in particular to $n=1$ and the proof above also applies to produce a family of irreducible representations of the subgroup $\PSL_2(\RR) \cong \Isom(\HC^1)^\circ$ into $\Isom(\HCI)$. On the other hand, we obtain a family of representations of $\PSL_2(\RR) \cong \Isom(\HR^2)^\circ$ into $\Isom(\HCI)$ by complexifying the representations into $\Isom(\HRI)$ studied in~\cite{Delzant-Py},\cite{Monod-Py}. These representations are also irreducible because Proposition~\ref{prop:complex:irred} below and its proof hold for $\Isom(\HR^n)^\circ$ exactly as they hold for $\Isom(\HR^n)$. (In any case we only passed to the connected component out of affection for $\PSL_2$.)

Finally, no element of the first family can be conjugated to any element of the second one because they have non-vanishing, respectively vanishing, pull-back of the K\"ahler class. This follows from Remark~\ref{rem:cohomologous}, respectively from the fact that the Cartan argument vanishes on the totally real subspace preserved by the second family.
\end{rem}

\section{General properties of representations}\label{sec:gen}
This section collects results that will be useful to organise the arguments towards a classification.

\subsection{Non-elementarity}
The first two statements are straightforward generalisations of similar statements in~\cite{Monod-Py2_arx},\cite{Monod-Py}.

\begin{prop}\label{prop:rep:length}
Let $\KK=\RR$ or $\CC$. Let $\ro\colon \Isom(\HK^\kappa) \to \Isom(\HC^{\kappa'})$ be a non\hyph{}elementary continuous representation, with $\kappa, \kappa'\geq 2$ arbitrary cardinals. If $g\in \Isom(\HK^\kappa)$ has positive translation length, then so does $\ro(g)$.
\end{prop}

\begin{proof}
Suppose for a contradiction that $\ro(g)$ has vanishing translation length. It cannot fix a point in $\HC^{\kappa'}$, for otherwise Lemma~\ref{lem:fixed:pt} together with a Cartan-like decomposition for $\Isom(\HK^\kappa)$ would imply that $\Isom(\HK^\kappa)$ has bounded orbits, contradicting non\hyph{}elementarity (compare Proposition~5.3 in~\cite{Monod-Py2_arx}). Thus $\ro(g)$ is parabolic and hence it has a \emph{unique} fixed point $\xi$ in $\partial\HC^{\kappa'}$. Let $P<\Isom(\HK^\kappa)$ be the stabiliser of the attracting fixed point of $g$ in $\partial\HK^\kappa$ and $\ol P$ the stabiliser of its repelling fixed point. A contraction group argument shows that $\ro(P)$ must fix $\xi$, and likewise for $\ro(\ol P)$; compare Proposition~2.1 in~\cite{Monod-Py}. These groups are however distinct maximal subgroups of $\Isom(\HK^\kappa)$ because the latter is doubly transitive on $\partial\HK^\kappa$. We conclude that $\Isom(\HK^\kappa)$ fixes $\xi$, a contradiction.
\end{proof}

\begin{cor}\label{cor:rest:n}
Let $\KK=\RR$ or $\CC$. Let $\ro\colon \Isom(\HK^\kappa) \to \Isom(\HC^{\kappa'})$ be a non\hyph{}elementary continuous representation, with $\kappa, \kappa'\geq 2$ arbitrary cardinals. Then the restriction of $\ro$ to any copy of $\Isom(\HC^{\kappa''})^\circ$ or of $\Isom(\HR^{\kappa''})^\circ$ in $\Isom(\HK^\kappa)$ with $\kappa''\geq 2$ remains non\hyph{}elementary.
\end{cor}

\begin{proof}
Upon restricting further, we just show that the restriction of $\ro$ to a copy of $\PSL_2(\RR)$ is non-elementary. In view of Lemma~\ref{lem:non-elem}, it suffices to find $g$ in $\PSL_2(\RR)$ such that $\ro(g)$ has positive translation length. By Proposition~\ref{prop:rep:length}, any hyperbolic element will do.
\end{proof}

\begin{rem}\label{rem:contain:n}
We now justify that the representations of $\Isom(\HC^n)$ constructed for Theorem~\ref{thm:rep:exist:n}, when restricted to $\Isom(\HR^n)$ with $n\geq 2$, contain the real representations of $\Isom(\HR^n)$ on $\HRI$ studied in~\cite{Monod-Py}. As in the above proof of Theorem~\ref{thm:rep:exist}, we combine Lemma~\ref{lem:pull-back:0} with Lemma~\ref{lem:tot:real} to deduce that the restriction to $\Isom(\HR^n)$ preserves some real hyperbolic subspace $\HR^{\kappa'}$ in $\HCI$. By Corollary~\ref{cor:rest:n}, this representation on $\HR^{\kappa'}$ remains non\hyph{}elementary and hence contains an irreducible (over $\RR$) part $\HR^{\kappa''}$. We can now apply the classification established in~\cite[Thm.~B]{Monod-Py} to $\HR^{\kappa''}$ and it remains only to justify that the parameter $t$ of that classification coincides with the parameter of the power construction used in Theorem~\ref{thm:rep:exist:n}; this follows from the characterisation of $t$ in terms of translation lengths.
\end{rem}

\subsection{Fixed points of point fixators}\label{sec:fixator}
The following proposition records an extended version of an argument used above in the proof of Theorem~\ref{thm:rep:exist}.

\begin{prop}\label{prop:O:unique}
Let $\ro$ be an irreducible continuous (real) linear representation of $\Isom(\HR^n)$ to $\OO(1, \infty)$, where $2\leq n \leq \infty$. Then $\ro(\OO(n))$ fixes a unique point in $\HRI$ and preserves a unique line in the underlying vector space.
\end{prop}

Note that the uniqueness of the fixed point is equivalent to the uniqueness of a $\ro(\Isom(\HR^n))$-equivariant map $\HR^n\to\HRI$, that is, to the uniqueness of an orbit isomorphic to $\HR^n$ as a homogeneous space. More generally, for non\hyph{}elementary representations, the \emph{canonical orbit} is defined to be the unique such orbit in the unique irreducible part. The corresponding comments hold for representations on $\HCI$ in view of Proposition~\ref{prop:unique:FP} below.

\begin{rem}\label{rem:canonical}
This terminology allows us to reformulate a statement obtained above in the proof of Theorem~\ref{thm:rep:exist}, namely the statement $p=p''$. This amounts precisely to the fact that the canonical orbit of $\Isom(\HCI)$ contains the canonical orbit of $\Isom(\HRI)$.
\end{rem}

\begin{proof}[Proof of Proposition~\ref{prop:O:unique}]
There is a fixed point by Lemma~\ref{lem:fixed:pt}. Its uniqueness follows from the classification of irreducible representations given in~\cite{Monod-Py},\cite{Monod-Py2_arx}. Specifically, when $n$ is finite, this is shown in Lemma~3.9 in~\cite{Monod-Py}. In the case $n=\infty$, see Remark~5.6 in~\cite{Monod-Py2_arx}. The statement about lines follows because the set of positive lines is open.
\end{proof}

We capture en passant the following approximative converse.

\begin{prop}\label{cor:O:unique:con}
Let $\ro$ be a non-elementary continuous (real) linear representation of $\Isom(\HR^n)$ to $\OO(1, \infty)$, where $2\leq n \leq \infty$. Suppose that $p\in \HRI$ is a point with total orbit under $\ro$.

If $p$ is the unique point fixed by $\ro(\OO(n))$, then $\ro$ is irreducible.
\end{prop}

\begin{proof}[Proof of Proposition~\ref{cor:O:unique:con}]
It suffices to prove that $\ro$ is minimal, i.e.\ that any invariant real hyperbolic subspace $\HR'\se\HRI$ contains $p$. This follows from the uniqueness assumption since $\HR'$ must contain a $\ro(\OO(n))$-fixed point by Lemma~\ref{lem:fixed:pt}.
\end{proof}

There is also a version of Proposition~\ref{prop:O:unique} for complex hyperbolic representations, though the proof is completely different and relies on abstract harmonic analysis. In particular it will be crucial for us that the proof does not rely on a classification.

Since this version can be stated for representations of $\Isom(\HKI)$ for $\KK$ either $\RR$ or $\CC$, we write simply $\Stab_\KK$ for the stabiliser in $\Isom(\HKI)$ of a point in $\HKI$, referring to Section~\ref{sec:isom}.

\begin{prop}\label{prop:unique:FP}
Let $\KK=\RR$ or $\CC$. Let $\ro\colon\Isom(\HKI)\to \Isom(\HCI)$ be an irreducible continuous representation. Then $\ro(\Stab_\KK)$ has a unique fixed point in $\HCI$.
\end{prop}

We emphasise that regardless of $\KK$, irreducibility is assumed over $\CC$ only. To prove this result, we shall reduce it to the following finite\hyph{}dimensional analogue, which is contained in Proposition~5.4 in~\cite{Monod-Py}.

\begin{prop}\label{prop:unique:FP:n}
Let $G$ be a locally compact group and $K<G$ a compact subgroup such that $(G,K)$ is a Gelfand pair. Let $\ro\colon G \to \Isom(\HCI)$ be an irreducible continuous representation. Then $\ro(K)$ has a unique fixed point in $\HCI$.\qed
\end{prop}

\begin{rem}\label{rem:unique:FP:n}
The fact that $G=\PO(n,1)$ and $K=\OO(n)$ form a Gelfand pair, and likewise in the complex case, can be found e.g.\ in~\cite{Faraut83}.
\end{rem}

\begin{proof}[Proof of Proposition~\ref{prop:unique:FP}]
For definiteness of the notation, we take $\KK=\RR$. We choose an increasing sequence
\begin{equation*}
\PO(n,1) \se \PO(n+1,1) \se \ldots \se \Isom(\HRI)
\end{equation*}
with dense union (as in~\cite[\S5.B]{Monod-Py2_arx}). The restriction of $\ro$ to $\PO(n,1)$ is non\hyph{}elementary by Corollary~\ref{cor:rest:n} and hence it admits a unique irreducible component, see Section~\ref{sec:non-el} above. We denote it by $X_n\se \HCI$. The uniqueness implies $X_n\se X_{n+1}$ and the irreducibility of $\ro$ on $\Isom(\HRI)$ implies that the union of all $X_n$ is dense in $\HCI$. In view of Proposition~\ref{prop:unique:FP:n}, there is in $X_n$ a unique $\OO(n)$-fixed point, which we denote by $p_n$.

In order to conclude the proof, it suffices to show that any $\OO$-fixed point $p$, which exists by Lemma~\ref{lem:fixed:pt}, satisfies $\lim_n p_n = p$. The projection of $p$ to $X_n$ coincides with $p_n$ by uniqueness of the latter and hence the conclusion follows by density of the union of all $X_n$.
\end{proof}

\subsection{Complexification}
One way to obtain representations to $\UU(1, \infty)$ and hence to $\Isom(\HCI)$ is to complexify a representation to $\OO(1, \infty)$. We record that this operation preserves irreducibility in the case of representations of $\Isom(\HR^n)$.

\begin{prop}\label{prop:complex:irred}
Let $\ro$ be a continuous (real) linear representation of $\Isom(\HR^n)$ to $\OO(1, \infty)$, where $2\leq n \leq \infty$.

If $\ro$ is irreducible, then its complexification is an irreducible (in the complex sense) representation of  $\Isom(\HR^n)$ to  $\UU(1, \infty)$.
\end{prop}

\begin{proof}[Proof of Proposition~\ref{prop:complex:irred}]
  Let $V$ be the underlying real topological vector space for $\ro$ and suppose for a contradiction that the complexified representation to $V\otimes\CC$ is not irreducible. We claim that $V$ admits a closed densely defined complex structure $J$ equivariant under $\Isom(\HR^n)$. Let indeed $W < V\otimes\CC$ be a proper invariant closed $\CC$-linear subspace. By considering its $\RR$-linear projections to the real and imaginary parts of $V\otimes\CC$, we see that $W$ is the graph of a densely defined ($\RR$-linear) operator $J$ in $V$. The fact that $W$ is stable under multiplication by $i$ forces $J^2=-1$ on its domain and $J$ is equivariant by construction. By Lemma~\ref{lem:fixed:pt} is there a point in the complex hyperbolic space associated to $W$ that is fixed by $\OO(n)$, noting that $W$ contains the unique irreducible part of $\ro$. It follows that $\OO(n)$ preserves a real line in $V$ under $\ro$ and this line is in the domain of $J$ by construction. Since $J$ is equivariant, $\OO(n)$ preserves also another real line, namely $JV$, which is different from $V$ since $J^2=-1$. This contradicts Proposition~\ref{prop:O:unique}.
\end{proof}

\section{Classification results}
\subsection{Rigidity of the argument}
\begin{flushright}
\begin{minipage}[t]{0.7\linewidth}\itshape\small
--- I came here for a good argument.\\
--- No you didn't, you came here for an argument.
\begin{flushright}
\upshape\small
Monty Python, 2 November 1972.
\end{flushright}
\end{minipage}
\end{flushright}
We begin with a few general observations. Conjugating a representation by an anti\hyph{}holomorphic isometry will reverse the sign of the Cartan argument and therefore the classification can ignore this sign, observing that $\Isom(\HR^\kappa)$ has a \emph{unique} non-trivial morphism to the group $\{\pm 1\}$.

In preparation for Theorem~\ref{thm:t:s}, let $0\leq t\leq 1$ and $s\in \RR$ be such that $(\beta^t, s \alpha)$ is of complex hyperbolic type, where $(\beta, \alpha)$ is the tautological kernel on $\HC^1$ given by $\cosh d$ and $\cartan$. Let $f\colon \HC^1\to \HCI$ be the corresponding continuous map.

Consider three points $x,y,z$ in $\HC^1$ forming an equilateral triangle of small side-length~$L$. As $L$ goes to zero, the area of $(x,y,z)$ behaves like $\sqrt 3 L^2/4$. Thus, in view of Remark~\ref{rem:omega:area}, the Cartan argument $\cartan(x,y,z)$ behaves like $\sqrt 3 L^2/2$ in absolute value. On the other hand, the image points $f(x),f(y),f(z)$ need not lie in a complex geodesic, nor indeed is there a canonical $2$-dimensional triangle spanned by them. They are still, however, mutually at equal distance, namely $L'$ determined by $\cosh L' = (\cosh L)^t$. Therefore, as $L'$ goes to zero, we have still an upper bound for the Cartan argument in the image, namely
\begin{equation*}
\limsup_{L'\to 0} \frac{\left|  \cartan(f(x),f(y),f(z))\right|}{{L'}^2} \leq \frac{\sqrt 3}{2}.
\end{equation*}
Therefore, if $\cartan(f(x),f(y),f(z))=s \cartan(x,y,z)$ holds for all $x,y,z$, then we deduce that $|s|$ is bounded by $\lim_{L\to 0} (L'/L)^2$. Using $\cosh L' = (\cosh L)^t$ we conclude that $|s|\leq t$ must hold.

Turning to the heart of the proof, we shall write $z^{t,s}$ for the \emph{mixed power} of a complex number $z$ in the slit plane $\CC\setminus \RR_{<0}$ defined by
\begin{equation*}
z^{t,s} = |z|^t\,e^{s \arg(z) \,i}.
\end{equation*}
Notice that this is a holomorphic function of $z$ only if $s=t$, and anti\hyph{}holomorphic only if $s=-t$.

\begin{proof}[Proof of Theorem~\ref{thm:t:s}]
Let $\zeta\neq 1$ be a third root of one. Given any complex number $z$, we claim that the matrix
\begin{equation*}
M(z)=\begin{pmatrix}
1 & z & \ol z\\
\ol z & 1 & z\\
z & \ol z & 1
\end{pmatrix}
\end{equation*}
is positive semi-definite if and only if $z$ belongs to the Euclidean triangle spanned by the powers of $\zeta$. In particular, only if $\Reel(z) \geq -1/2$. The claim follows from Sylvester's criterion because we have the factorisation
\begin{equation*}
\det M(z) = (2a + 1) (a-1 + \sqrt 3 b) (a-1 - \sqrt 3 b) 
\end{equation*}
for $z=a+i b$ and because the principal minors are $1-|z|^2$.
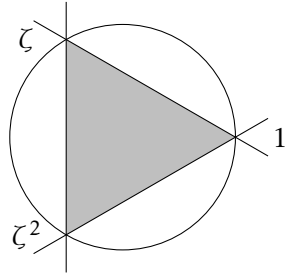
\begin{center}
\begin{tikzpicture}[scale=1.5]
\fill[fill=black!25] (-0.5, 0.866) -- (-0.5, -0.866) -- (1,0);
\draw (0,0) circle (1);
\draw (-0.5,1.2) -- (-0.5,-1.2);
\draw (1.289,-0.167) -- (-0.789,1.033);
\draw (1.289,0.167) -- (-0.789,-1.033);
\draw (-0.85,0.866) node {$\zeta$};
\draw (-0.85,-0.866) node {$\zeta^2$};
\draw (1.4,0) node {$1$};
\end{tikzpicture}
\end{center}
For every $\epsi>0$, we define four points of $\HC^1$ by
\begin{equation*}
x_0=1\oplus 0\kern5mm \text{and} \kern5mm x_j= (1+\epsi)^{1/2} \oplus \epsi^{1/2} \zeta^j
\end{equation*}
when $j=1,2,3$. Notice that we are exactly in the model considered in Remark~\ref{rem:proof:KCHT:Phi} with here $X_j=x_j$. Specifically, for $j,k\geq 1$ we have $\beta(x_j, x_0)= (1+\epsi)^{1/2}$ and 
\begin{equation}\label{eq:t:s:B}
e^{i \alpha(x_0, x_j, x_k)}\beta(x_j, x_k)= B(x_j, x_k) = 1+ \epsi  - \epsi\, M(\zeta)_{j,k}.
\end{equation}
Applying Definition~\ref{def:KCHT} to $(\beta^t, s \alpha)$, the following $3$-by-$3$ matrix must be positive semi-definite:
\begin{equation*}
(1+ \epsi)^t - e^{i s \alpha(x_0, x_j, x_k)} \, \beta(x_j, x_k)^t \kern10mm(j,k\geq 1).
\end{equation*}
Using~\eqref{eq:t:s:B}, this matrix is
\begin{equation*}
(1+\epsi)^t \,\one\one\ctran - \left(1+ \epsi  - \epsi \, M(\zeta) \right)^{t,s}.
\end{equation*}
The diagonal terms being $(1+\epsi)^t -1$, this matrix can be expressed as
\begin{equation*}
\left((1+\epsi)^t -1\right)
\,\,M\left( \frac{(1+\epsi)^t - \left(1+ \epsi -  \epsi \zeta \right)^{t,s}}{(1+\epsi)^t -1} \right).
\end{equation*}
In particular, the positivity condition $\Reel(z) \geq -1/2$ for $M(z)$ implies, after regrouping terms:
\begin{equation}\label{eq:real:cond}
2\Reel\left(1+ \epsi -  \epsi \zeta \right)^{t,s} - 3 (1+\epsi)^t \leq -1.
\end{equation}
We claim that this inequality holding for all $\epsi>0$ forces $s\geq t$. Indeed, consider the left hand side of the inequality~\eqref{eq:real:cond} as a function $R(\epsi)$ of the parameter $\epsi$, noting $R(0)=-1$. To perform the computation of the mixed power, observe that $1+ \epsi -  \epsi \zeta$ has modulus $(1+3\epsi + 3 \epsi^2)^{1/2}$ and argument $\pm \arctan(\sqrt3 \epsi /(2+3\epsi))$, the sign depending on our choice of $\zeta$. Thus we compute $R'(0)=0$. Therefore, considering~\eqref{eq:real:cond} for $\epsi>0$ small enough, we must have $R''(0)\leq 0$. A further direct computation gives $R''(0) = 3t^2 - 3s^2$. It thus follows $|s|\geq t$ and the proof is complete because we already know $|s|\leq t$.
\end{proof}

\subsection{Proof of Theorem~\ref{thm:rep:classify}}
Consider an arbitrary irreducible self\hyph{}representation $\ro$ of $\Isom(\HCI)$. The proof given in~\cite{Monod-Py2_arx} for the automatic continuity of irreducible self\hyph{}representations of $\Isom(\HRI)$ can be applied as is to deduce the corresponding statement for $\Isom(\HCI)^\circ$ and hence for $\Isom(\HCI)$. Indeed, the preliminary Proposition~6.1 in~\cite{Monod-Py2_arx} holds with the same proof, and the results of Ricard--Rosendal~\cite{Ricard-Rosendal} and of Tsankov~\cite{Tsankov13} that we cited for $\OO$ are also proved for $\UU$ in the same references.

We choose an origin $q\in \HRI\se \HCI$ and denote by $\tUU$ and $\OO$ its stabiliser in $\Isom(\HCI)$ and in $\Isom(\HRI)$ respectively, see Section~\ref{sec:isom}. Since $\ro$ is irreducible, Proposition~\ref{prop:unique:FP} implies that there is a unique point $p\in\HCI$ fixed by $\ro(\tUU)$. Let $f_0\colon \HCI\to\HCI$ be the corresponding $\ro$-equivariant map; in particular $p=f_0(q)$.

The restriction of $\ro$ to $\Isom(\HRI)$ is non\hyph{}elementary by Corollary~\ref{cor:rest:n} and thus admits a unique irreducible part $\HC^{\lambda} \se \HCI$. The assumption made in Theorem~\ref{thm:rep:classify} guarantees that the point $p$ belongs to $\HC^{\lambda}$. Combining Lemma~\ref{lem:pull-back:1} with Lemma~\ref{lem:tot:real}, we deduce that the restriction of $\ro$ to $\Isom(\HRI)$ preserves a real hyperbolic subspace $\HR^{\lambda'}$ in $\HC^{\lambda}$. Its $\RR$-irreducible part $\HR^{\lambda''}\se \HR^{\lambda'}$ can be described thanks to the classification given in Theorem~II of~\cite{Monod-Py2_arx}. Namely, there is $0<t\leq 1$ and a $\ro(\Isom(\HRI))$-equivariant map $f\colon \HRI\to \HR^{\lambda''}$ such that
\begin{equation}\label{eq:dist:on:HR}
\cosh d\big(f(x), f(y)\big) = \big(\cosh d(x,y)\big)^t
\end{equation}
holds for all $x,y\in\HRI$.

Now we use the fact that $p$ is the unique $\ro(\OO)$-fixed point in $\HC^{\lambda}$ according to Proposition~\ref{prop:unique:FP}. Thus $f(q)=p=f_0(q)$ and it follows that $f$ is the restriction of $f_0$ to $\HRI$. In view of the metric double transitivity of $\HCI$, the relation~\eqref{eq:dist:on:HR} implies that the same relation holds for all $x,y\in\HCI$, with now $f_0$ instead of $f$. In other words, the invariant kernel of complex hyperbolic type $(\beta, \alpha)$ on $\HCI$ associated to $\ro$ satisfies $\beta(x,y)=(\cosh d(x,y))^t$. In order to conclude, using Theorem~\ref{thm:KCHT-char}, that $\ro$ is conjugated to the representation of Theorem~\ref{thm:rep:exist} determined by $t$, it only remains to justify that $\alpha$ is given by $\alpha = \pm t \cartan$, recalling that the sign can be changed by conjugating with an anti\hyph{}holomorphic isometry of the target. This last step, however, follows from Theorem~\ref{thm:t:s} applied to any complex geodesic line. This completes the proof of Theorem~\ref{thm:rep:classify}.\qed

\section{Trees and their kernels}
\subsection{The distance function as a kernel}\label{sec:tree}
Using a basic explicit embedding, Haagerup showed in 1979 that the distance function on a simplicial tree is of conditionally negative type (Lemma~1.2 in~\cite{Haagerup79}). In fact his focus was on free groups and his goal was to show that for all $\lambda\geq 1$ the kernel
\begin{equation}\label{eq:tree:UH}
(x,y) \longmapsto \lambda^{-d(x,y)}
\end{equation}
is of positive type, which he deduced immediately from the conditionally negative type statement using Schoenberg's theorem. This result launched the entire theory of what is now called the Haagerup approximation property~\cite{Cherix-Cowling-Jolissaint-Julg-Valette}.

\smallskip
The content of Proposition~\ref{prop:tree} is that the kernel
\begin{equation}\label{eq:tree:hyp}
(x,y) \longmapsto \lambda^{d(x,y)}
\end{equation}
is of hyperbolic type. The glaring symmetry between~\eqref{eq:tree:UH} and~\eqref{eq:tree:hyp} will be discussed in the general setting in Section~\ref{sec:inverse} below; in particular, our statement on~\eqref{eq:tree:hyp} formally implies Haagerup's result on~\eqref{eq:tree:UH}.

\begin{proof}[Proof of Proposition~\ref{prop:tree}]
Let $(X, d)$ be a metric tree, fix $n\in \NN$ and choose $n+1$ points $x_0, x_1, \ldots, x_n\in X$. We argue by induction on the combinatorial type of the sub-tree spanned by the $x_j$. Let thus $c_1, \ldots, c_n\in\RR$; we need to show that
\begin{equation}\label{eq:tree}
\sum_{j,k\geq 1} c_j c_k \left( \lambda^{d(x_j, x_0) + d(x_0, x_k)} - \lambda^{d(x_j, x_k)} \right) \geq 0.
\end{equation}
The term in brackets vanishes whenever $x_0$ separates $x_j$ and $x_k$ in $X$. Therefore, the sum~\eqref{eq:tree} splits into a number ($\geq 1$) of similar sums where $x_0$ does not separate any of the pairs of other points. We can therefore assume that $x_0$ is a leaf of the sub-tree spanned by all $x_j$. We can moreover reduce immediately to the case where all $x_j$ with $j\neq 0$ are distinct from $x_0$. Thus, being a leaf, $x_0$ does not belong to the sub-tree $Y\se X$ spanned all $x_j$ with $j\neq 0$.

We now consider the flow $x_0^{(t)}$, defined for $t\geq 0$ small enough, where $x_0^{(t)}$ is the unique point at distance $t$ from $x_0$ on the geodesic from $x_0$ to $Y$. We denote by $\Sigma(t)$ the sum~\eqref{eq:tree} with $x_0^{(t)}$ in lieu of $x_0$. Let $t_\mathrm{max}>0$ be the time at which $x_0^{(t)}$ joins the sub-tree $Y$. We know by induction that $\Sigma(t_\mathrm{max})\geq 0$; therefore it is sufficient to establish that $\Sigma$ has a non\hyph{}positive derivative $\dot\Sigma$ with respect to $t$ on $(0, t_\mathrm{max})$. The term in the bracket of the sum for $\Sigma(t)$ is
\begin{equation*}
\lambda^{d(x_j, x_0) -t + d(x_0, x_k) -t} - \lambda^{d(x_j, x_k)}.
\end{equation*}
This implies
\begin{equation*}
\dot\Sigma(t) = \sum_{j,k\geq 1} c_j c_k \left( -2 \log \lambda \, \lambda^{d(x_j, x_0) -t + d(x_0, x_k) -t}  \right).
\end{equation*}
This can be re-written as
\begin{equation*}
-2 \log \lambda \, \sum_{j,k\geq 1} \left (c_j \lambda^{d(x_j, x_0) -t }\right) \left(  c_k \lambda^{ d(x_0, x_k) -t}  \right)  = -2 \log \lambda \,\left(  \sum_{j\geq 1} c_j \lambda^{d(x_j, x_0) -t }\right)^2
\end{equation*}
which is indeed non\hyph{}positive as claimed.
\end{proof}

\subsection{The $\RR$-trees of $\GL_2$}
Recall that the field $\CC$ admits countable (non-discrete) valuations~$\nu$; for instance, given a prime~$q$, any isomorphism of $\CC$ with an algebraic closure of $\QQ_q$ gives such a valuation, whose value group we realise as a countable subgroup of~$\RR$. Following Tits~\cite[\S5]{Tits77}, this provides us with an isometric action of $\GL_2(\CC)$ on an $\RR$-tree $(X_\nu, d_\nu)$, which descends to an action of $\PGL_2(\CC)$; we refer to the detailed exposition of~\cite{Alperin-Bass}. Writing $A\se \CC$ for the valuation ring of $\nu$, we choose the lattice $L=A\oplus A$ as a base-point for $X_\nu$. We have for instance $d_\nu(L, L') = \nu(a)$ if $L'=(a A)\oplus A$ with $a\in A$. Moreover, the element $s=\left(\begin{smallmatrix}q & 0\\ 0 & 1\end{smallmatrix}\right)$ has translation length $\ell(s)=1$, see~\cite[(B.7)]{Alperin-Bass}.

We now fix $\lambda> 1$. Combining Proposition~\ref{prop:tree} with Proposition~\ref{prop:CHT:char} and Theorem~\ref{thm:R:rep}, we obtain a representation $\ro$ of $\PGL_2(\CC)$ on $\HR^\kappa$ for some $\kappa$. More precisely, we have a $\ro$-equivariant map $f\colon X_\nu\to \HR^\kappa$ such that
\begin{equation}\label{eq:tree:H}
\cosh d\big(f(x), f(y)\big) = \lambda^{d_\nu(x,y)} \kern10mm(\forall\,x,y\in X_\nu).
\end{equation}
If we combine $\ell(s)=1$ with the formula for translation lengths given in~\eqref{eq:length:F}, then the relation~\eqref{eq:tree:H} implies $\ell(\ro(s))=\log\lambda>0$. In particular $\ro(s)$ is hyperbolic and $\ro$ is non\hyph{}elementary by Lemma~\ref{lem:non-elem}. At this point we can consider the unique irreducible part $\HR^\eta$ of $\ro$. Since $\PGL_2(A)$ fixes the point $L\in X_\nu$, it fixes some point in $\HR^\eta$, and the $\PGL_2(\CC)$-orbit of that point is countable since the value group is countable. The case of $\PGL_2(\RR)$ is obtained after taking again the unique irreducible part of the restriction of $\ro$ to $\PGL_2(\RR)$; for this, we need to know that the restriction is non\hyph{}elementary, which follows as above from Lemma~\ref{lem:non-elem} because $s$ is contained in $\PGL_2(\RR)$.

To complete the proof of Theorem~\ref{thm:wild}, it only remains to observe that $\eta$ cannot be finite; one way to see this (maybe not the simplest) is to invoke~\cite{Borel-Tits73}.\qed

\begin{rem}\label{rem:compare:local}
It is instructive to contrast the geometric properties of the tree representations of~\cite{Burger-Iozzi-Monod} and of the above section against the smooth representations of Theorem~\ref{thm:rep:exist}, Theorem~\ref{thm:rep:exist:n} and their real counterparts from~\cite{Monod-Py},\cite{Monod-Py2_arx}. To this end, let us denote by $d_0$ the ``source'' distance on the trees or hyperbolic spaces, and by $d_1$ the ``target'' hyperbolic distance. When $d_0 \to \infty$, we find in all cases that $d_1$ approaches $m \, d_0 + h$ for some constants $m, h\geq 0$. For instance, $m=\log\lambda$ in the notation of the tree case and $m=t$ for the representations of~\cite{Monod-Py} and of Theorem~\ref{thm:rep:exist}.

By contrast, the behaviour for $d_0\to 0$ presents radical differences. It was observed in~\cite{Monod-Py},\cite{Monod-Py2_arx} that the irreducible representations of real hyperbolic groups give locally again a linear growth, now of the form $d_1\sim m' \,d_0$ with $m'=  \sqrt t$ in the infinite\hyph{}dimensional case and $m'=(t(n+t-1)/n)^{1/2}$ in dimension $n$. On the other hand, such a tame local behaviour is incompatible with the branching of the trees, where we have $d_1\sim m'' \, \sqrt{d_0}$ with $m''=\sqrt{2 \log \lambda}$. As we shall see in Section~\ref{sec:free}, there are also advantages to adapting this branching to the case of hyperbolic spaces.
\end{rem}

\section{Free products of real hyperbolic spaces}\label{sec:free}
\subsection{Anisotropic deformations}
Our work with P.~Py~\cite{Monod-Py2_arx} classifies all irreducible self\hyph{}representations of $\Isom(\HRI)$, showing that the corresponding maps $\HRI\to\HRI$ are defined by the power kernels. We now want to discuss what natural maps $\HRI\to\HRI$ are still left \emph{besides} this classification. In other words, maps $\HRI\to\HRI$ that are equivariant for a self\hyph{}representation $\ro$ which is \emph{not irreducible}.

\smallskip
The most complete breakdown of irreducibility is the case of elementary representations. In that situation, either $\ro$ fixes point in $\HRI$, or it fixes point or pair in $\partial\HRI$. The former case reduces completely to the study of orthogonal representations, a classical topic that we shall not revisit here. The latter cases can be analysed in terms of a Busemann characters together with affine isometric actions, again a classical topic; this reduction is briefly recalled in Section~\ref{sec:FP:infty} below.

\medskip\itshape
We are thus left with representations $\ro$ that are non\hyph{}elementary but not irreducible.\upshape\ We recall that such a representations has a unique irreducible part together with a complement that is anisotropic (i.e.\ orthogonal), see~\cite[\S4]{Burger-Iozzi-Monod}. Geometrically, $\HRI$ contains a unique minimal $\ro$-invariant hyperbolic subspace, on which the action is indeed described by the classification. However, there is a transversal part described by the orthogonal representation, and this \emph{anisotropic deformation} can give rise to interesting maps $\HRI\to\HRI$, or equivalently to interesting kernels of real hyperbolic type that are not powers of the tautological kernel $\cosh d$.

\smallskip
The following lemma records a general construction of anisotropic deformations.

\begin{lem}\label{lem:anis}
Let $\beta$ be a kernel of real hyperbolic type on a set $X$. Let $\Phi$ be a (real) kernel of positive type on $X$ taking the constant value~$1$ on the diagonal. Then
\begin{equation*}
\beta' = \cosh^2\delta\cdot\beta  - \sinh^2\delta\cdot \Phi
\end{equation*}
is a kernel of real hyperbolic type on $X$ for all $\delta\geq 0$.
\end{lem}

Rather than verifying an inequality, we shall prove this by inspecting the associated geometric situation so that it will become clear that this formula corresponds to the anisotropic deformations described in the paragraphs above. 
The parameter $\delta$ turns out to be the (constant) distance between the original and the deformed images of $X$.

\begin{proof}[Proof of Lemma~\ref{lem:anis}]
Associated to $\beta$ we have a map $f\colon X\to \RR\oplus H$ to the upper sheet of the hyperboloid defined by the Minkowski form $B$ (this corresponds to the real case of Remark~\ref{rem:proof:KCHT:Phi}). Let $h\colon X\to V$ be the map to a Hilbert space $V$ associated to $\Phi$ by the GNS construction. Consider the Minkowski space $\RR\oplus H\oplus V$ with the form $B'=B-\her\cdot\cdot$. Finally, define $f'\colon X\to \RR\oplus H\oplus V$ by
\begin{equation*}
f' =  (\cosh\delta\cdot f)  \oplus (\sinh\delta\cdot h)
\end{equation*}
By definition, $f'$ ranges in the upper sheet of the hyperboloid for $B'$ and the associated kernel is $\beta'$; moreover, $d(f, f') = \delta$.
\end{proof}

\subsection{A deformation of $\HR^\kappa$}
We illustrate the use of this construction with an anisotropic deformation of $\HR^\kappa$, interesting even for $\kappa$ finite, enjoying the following two properties. On the large scale the embedding is essentially isometric, but on the small scale it is as singular as the tree embeddings (Remark~\ref{rem:compare:local}). This is all encapsulated in the kernel, as follows.

\begin{prop}\label{prop:exp}
The kernel $(x,y) \mapsto e^{d(x,y)}$ is of real hyperbolic type on $\HR^\kappa$.
\end{prop}

\begin{proof}
Since the tautological kernel $\beta$ is $\cosh d$, Lemma~\ref{lem:anis} with $\delta=\arcosh\sqrt2$ shows that it suffices to justify that $\Phi= e^{-d}$ is of positive type on $\HR^\kappa$. This is a known fact for the following reason. By Schoenberg's theorem, it suffices to show that $d$ is of conditionally negative type, which was proved in Corollaire~7.4 of~\cite{Faraut-Harzallah74}; there is no harm here in assuming $\kappa$ finite since the condition regards finitely many points at a time.
\end{proof}

Combining Proposition~\ref{prop:exp} with Theorem~\ref{thm:R:power}, we deduce that $\HR^\kappa$ admits kernels analogous to the kernels on trees of Proposition~\ref{prop:tree}.

\begin{cor}\label{cor:exp}
The the kernel $\lambda^d$ is of real hyperbolic type on $\HR^\kappa$ for all 
$1\leq \lambda\leq e$.\qed
\end{cor}

\begin{rem}
In contrast to the case of trees, the restriction $\lambda\leq e$ is necessary here (unless $\kappa=1$, in which case $\HR^\kappa$ is an elementary $\RR$-tree). Indeed, consider a map $f\colon \HR^\kappa \to \HR^{\kappa'}$ with $\cosh d(f(x), f(y)) = \lambda^{d(x,y)}$ for all $x,y\in \HR^\kappa$. Note that $f$ is quasi-isometric and hence extends to a map $\partial f\colon \partial\HR^\kappa \to \partial\HR^{\kappa'}$. If we express the angular metric at infinity in terms of exponentials of distances (see equation~\eqref{eq:angular:metric} below), we see that for $\xi, \eta\in \partial\HR^\kappa$ the distance between $\partial f(\xi)$ and $\partial f(\eta)$ is of the order of the distance between $\xi$ and $\eta$ to the power $\log \lambda$. When $\lambda>e$, this power is~$>1$ and the triangle inequality in the image becomes incompatible with the fact that the source contains smoth arcs (thanks to $\kappa\geq 2$).

The reader wishing to make this computation explicit can use the hyperbolic law of cosines to verify
\begin{equation}\label{eq:angular:metric}
\sin\tfrac12 \angle_p(\xi,\eta) = \lim_{x,y} \, e^{-\frac12(d(p,x) + d(p,y) - d(x,y))}
\end{equation}
where $\angle_p$ is the angular metric with respect to a base-point $p$ and $x,y$ converge radially to $\xi$ and $\eta$ respectively.
\end{rem}

\subsection{Glueing}
We recall the definition of the glueing of two metric spaces $X$ and $Y$, denoting all distance functions by $d$. Given two points $x_0\in X$ and $y_0\in Y$, the glueing $X\vee Y = X\underset{x_0=y_0}{\vee} Y$ is the quotient of the disjoint union $X\sqcup Y$ obtained by identifying $x_0$ with $y_0$. The distance on $X\vee Y$ is defined to coincide with the original distances on $X$ and on $Y$, and for mixed pairs it is defined by
\begin{equation*}
d(x,y) = d(x, x_0) + d(y_0, y) \kern10mm(x\in X, y\in Y).
\end{equation*}
We use the term of \emph{free products} of metric spaces somewhat informally to refer to the following situation. We are given a family $(X_j, x_{j,0})_{j\in J}$ of pointed metric spaces. In the interest of keeping a lower topological weight, hence a lower (infinite) hyperbolic dimension, we can specify subsets $X'_j\se X_j$ of admissible glueing points, for instance a countable orbit of an isometry group of interest. No restriction corresponds to the choice $X'_j=X_j$.

Starting from some $X_j$, we glue copies of each $X_k$ with $k\neq j$ at each $x_0\in X'_j$ by identifying $x_0$ with $x_{k,0}$. We then repeat this operation on each new space that was attached in a transfinite sequence of steps conducted over all countable ordinals. Finally, the free product is by definition the subset of points at finite total distance from the original copy of $X_j$.

\medskip

The precise construction is not very important for our present purposes, because the proof of the following result only necessitates to consider glueings of finitely many (but of course arbitrarily many) spaces.

\begin{thm}\label{thm:free}
Let $1< \lambda \leq e$.\\
Any free product $(X, d)$ of real hyperbolic spaces admits a natural embedding $f\colon X \to \HR^\kappa$ (for some $\kappa$) such that
\begin{equation*}
\cosh d(f(x), f(y)) = \lambda^{d(x,y)}
\end{equation*}
holds for all $x,y\in X$. Moreover, this embedding is equivariant for a representation of $\Isom(X)$ to $\Isom(\HR^\kappa)$.
\end{thm}

The idea behind this construction is that each hyperbolic space embedded using the kernel of Corollary~\ref{cor:exp} will be sufficiently ``pointy'' at each of its points to allow a glueing with each of the other embedded spaces.

\begin{proof}
Given the general kernel yoga, we only need to prove that for any \emph{finite} glueing of hyperbolic spaces, the kernel $\lambda^d$ is of real hyperbolic type. We shall show this by induction on the number of hyperbolic spaces that have been glued together. The base case of a single space is given by Corollary~\ref{cor:exp}.

For the inductive steps, we consider a glueing $X \vee Y$ such that on both $X$ and $Y$ the kernel $\lambda^d$ is of real hyperbolic type. Let $x_0\in X$ and $y_0\in Y$ be the points that have been identified in $X \vee Y$. In order to prove that $\lambda^d$ is of real hyperbolic type on the whole $X \vee Y$, we shall use our freedom to choose a base-point in Proposition/Definition~\ref{propdef:KHT} (compare also Remark~\ref{rem:proofs:KCHT}). We choose the glueing point $x_0=y_0$ and hence we must show that the kernel
\begin{equation*}
(z,z') \longmapsto \lambda^{d(z, x_0) + d(x_0, z')} - \lambda^{d(z, z')}
\end{equation*}
is of positive type. By the definition of glueings, the above expression vanishes unless $z$ and $z'$ lie in the same component $X$ or $Y$. Therefore, as in the proof of Proposition~\ref{prop:tree}, the double sum defining the condition of positive type splits into two double sums and for each one we are reduced to our inductive assumption.
\end{proof}

We observe that the induction step established criterion~\eqref{pt:KHT:some} of Propo\-si\-tion/Defi\-ni\-tion~\ref{propdef:KHT} by reducing it to the equivalent but formally stronger criterion~\eqref{pt:KHT:KPT}.

\section{Additional considerations for real kernels}
\subsection{Fixed points at infinity}\label{sec:FP:infty}
In order to establish a sufficient condition for irreducibility, we first analyse a case of elementary actions that we will then want to rule out.

Suppose that an isometric action of a group $G$ on $\HR^\kappa$ fixes a point at infinity and let $\chi\colon G\to \RR$ be the corresponding Busemann character. Recall that $\chi$ is defined by the property that if $b$ is a Busemann function at the fixed point, then $b(gx)=b(x)+\chi(g)$ holds for all $x\in \HR^\kappa$. (This is defined in the \cat0 generality, a Busemann function being a limit of functions $x\mapsto d(x,x_n) - d(x_0, x_n)$ for any sequence $x_n$ converging to the fixed point at infinity.)

Consider now the function of real hyperbolic type $F$ on $G$ defined by some point $p\in \HR^\kappa$ by $F(g) = \cosh d(g p, p)$. We seek a property of $F$ that reflects the fact that $G$ has a fixed point at infinity:

\begin{prop}\label{propKHT:para}
The function $\Psi\colon G\to\RR_{\geq 0}$ determined by the identity
\begin{equation}\label{eq:KHT:para}
F(g) = \cosh \chi(g) + e^{-\chi(g)} \, \Psi(g)
\end{equation}
is a function of conditionally negative type on $G$.
\end{prop}

\begin{proof}
Let $\xi\in\partial\HR^\kappa$ be the given point fixed by $G$. Recall from~\cite[\S2]{Monod-Py2_arx} that we can parametrise $\HR^\kappa$ by $\RR\times E$, where $E$ is a Hilbert space, under the map
\begin{equation}\label{eq:horopara}
\sigma_s(v) = \begin{pmatrix}
\frac 12 (e^s + e^{-s} \|v\|^2)\\
e^{-s}\\
e^{-s} v
\end{pmatrix}
\end{equation}
with $s\in\RR$ and $v\in E$ so that $-s$ is a Busemann function at $\xi$ applied to $\sigma_s(v)$. We recall further that in the above model the Minkowski from $B$ is given by
\begin{equation*}
B\left(
\begin{pmatrix}
x\\
y\\
v
\end{pmatrix}
\begin{pmatrix}
x'\\
y'\\
v'
\end{pmatrix}
\right)
=xy' + x'y - \her{v}{v'}.
\end{equation*}
Moreover we can normalise the parametrisation so that the point $p\in \HR^\kappa$ defining $F$ is $p= \sigma_0(0)$. Now $g\in G$ acts by $g\sigma_s(v) = \sigma_{s-\chi(g)}(g.v)$ for some action $g.v$ on $E$. This action $g.v$ is an affine action with multiplicative character $e^{\chi}$; in other words, $g_*v= e^{\chi(g)} g.v$ defines an affine isometric action on $E$ and thus the function $\Psi$ defined on $G$ by $\Psi(g) = \tfrac12 \|g. 0\|^2$ is a function of conditionally negative type.

Finally, it remains only to compute
\begin{equation*}
F(g)= B\left(g \sigma_0(0),  \sigma_0(0)\right) = B\left(
\begin{pmatrix}
\tfrac12 e^{-\chi(g)} + \tfrac12 e^{\chi(g)} \|g.0\|^2\\
 e^{\chi(g)} \\
e^{\chi(g)} g.0
\end{pmatrix}
\begin{pmatrix}
\tfrac12\\
1\\
0
\end{pmatrix}
\right).
\end{equation*}
After simplification, we find precisely~\eqref{eq:KHT:para}, as was to be shown.
\end{proof}

\begin{rem}\label{rem:KHT:para}
For later use, we record that for $p'= \sigma_1(0)$ and $F'(g) = B(gp', p')$, the same computations as above lead to formula~\eqref{eq:KHT:para} with $\Psi$ replaced by $\Psi' = e^{-2} \Psi$.
\end{rem}

\subsection{Functions defining irreducible representations}
We can now establish the irreducibility criterion.

\begin{proof}[Proof of Proposition~\ref{prop:min:irred}]
Consider an isometric action of $G$ on some $\HR^\kappa$ and $p\in \HR^\kappa$ such that $G p$ is total and such that $F(g) = \cosh d(gp, p)$ for all $g\in G$.

We can assume that $p$ is not fixed by $G$ and does not belong to a geodesic line preserved by $G$, since otherwise we would have $\alpha\leq 1$ and the conclusion would hold for trivial reasons.

We claim that $G$ does not fix a point at infinity. Indeed, otherwise Proposition~\ref{propKHT:para} implies that $F$ is of the form
\begin{equation*}
F(g) = \cosh \chi(g) + e^{-\chi(g)} \, \Psi(g).
\end{equation*}
Moreover, by Remark~\ref{rem:KHT:para}, another point would provide us with a function of hyperbolic type $F'$ given by
\begin{equation*}
F'(g) = \cosh \chi(g) + e^{-\chi(g)} \, e^{-2}\,\Psi(g).
\end{equation*}
We see that $F'\leq F$ and that the ratio $F'/F$ is bounded by~$e^2$. Therefore the assumption on $F$ implies $F'=F$, or equivalently that $\Psi$ vanishes. This means that $G$ preserves in fact an entire geodesic line; under the parametrisation~\eqref{eq:horopara}, this is precisely the geodesic $s\mapsto \sigma_s(0)$, which contains $p$. This contradiction proves the claim.

It now follows that $\HR^\kappa$ contains a minimal $G$-invariant real hyperbolic subspace $X\se \HR^\kappa$, see Proposition~4.3 of~\cite{Burger-Iozzi-Monod}. It remains to show $X=\HR^\kappa$. For this, it suffices to show $p\in X$ since $Gp$ is total. Consider the nearest-point projection $\pi\colon \HR^\kappa\to X$ and define the function of hyperbolic type $F'$ by $F'(g) = \cosh d(g \pi(p), \pi(p))$. Since $\pi$ is $G$-equivariant, we have
\begin{equation*}
F'(g) = \cosh d\big(\pi(g p), \pi(p)\big).
\end{equation*}
Since projections are non\hyph{}expanding~\cite[II.2.4(4)]{Bridson-Haefliger}, we have $F'\leq F$. On the other hand, we also have
\begin{equation*}
d(gp, p) - d\big(\pi(g p), \pi(p)\big) \leq d\big(qp, \pi(g p)\big) + d\big(\pi(p), g\big) = 2 d\big(\pi(p), g\big).
\end{equation*}
This implies that $\cosh d(gp, p)$ is bounded by a multiple of $\cosh d(\pi(g p), \pi(p))$ and thus our assumption now implies $F'=F$.

We can assume that some $g\in G$ moves $p$ since otherwise $\HR^\kappa$ is reduced to a point anyway. Now we have $p= \pi(p)$ because otherwise the sandwich lemma (see~\cite[II.2.12]{Bridson-Haefliger})  would force
$$d\big(g \pi(p), \pi(p)\big) < d(gp, p)$$
which would contradict $F'=F$. Thus $p\in X$, as was to be established.
\end{proof}

\subsection{Epilogue on the definition of kernels}\label{sec:inverse}
To conclude these notes, we return to the definition of kernels and functions of real hyperbolic type.

First, we record yet another characterisation, recalling from Propo\-si\-tion/Defi\-ni\-tion~\ref{propdef:KHT} that a symmetric kernel $\beta\colon X\times X\to\RR_{\geq 0}$ taking the constant value~$1$ on the diagonal is of real hyperbolic type iff for some or every $x_0\in X$, the kernel
\begin{equation*}
\Phi(x,y) = \beta(x, x_0) \,\beta(x_0, y) -  \beta(x,y)
\end{equation*}
is of positive type. We claim that this is equivalent to the following:

For some or every $x_0$, the kernel
\begin{equation*}
\Psi(x,y) = \tfrac12\left( \beta(x,x_0) - \beta(x_0,y)\right)^2 +  \beta(x,y)-1
\end{equation*}
is of conditionally negative type. Indeed, a computation shows
\begin{equation*}
\Phi(x,y) = \Psi(x, x_0)+ \Psi(x_0, y) - \Psi(x,y)
\end{equation*}
and now the equivalence becomes a standard polarisation identity, see Lem\-ma~C.3.1 in~\cite{Bekka-Harpe-Valette} (that lemma applies because $\Psi$ vanishes on the diagonal).

\medskip

Next, we return to the observation that the two kernels at the beginning of Section~\ref{sec:tree} are inverse to each other. This is a manifestation of the following general fact: if $\beta$ is of real hyperbolic type, then $1/\beta$ is of positive type.

Indeed, more is true: $\beta^{-t}$ is a kernel of positive type for all $t\geq 0$. Indeed, by Schoenberg's theorem~\cite{Schoenberg38}, that statement is equivalent to the conditionally negative type of $\log \beta$. For the latter, it suffices by Proposition~\ref{propdef:KHT} to show that for any hyperbolic space $\HR^\kappa$, the kernel
$$\log \cosh d \colon \HR^\kappa\times \HR^\kappa \lra \RR$$
is of conditionally negative type. This is the content of Proposition~7.3 in~\cite{Faraut-Harzallah74}; the latter is stated with $\kappa$ finite, but its proof does not use this restriction (which is irrelevant anyway since the conditionally negative type condition is checked on finitely many points at a time). 

\medskip
Thus $1/\beta$ is a kernel of positive type with the special property that \emph{all its positive Hadamard powers} remain of positive type; this is called an \emph{infinitely divisible} kernel in the sense of Loewner. We recall that in general any non\hyph{}integer power fails to preserve positive type (Theorem~2.2 in~\cite{FitzGerald-Horn}). In fact infinitely divisible kernels have been characterised as follows. Suppose that $\gamma$ is of positive type with non\hyph{}negative entries. Then $\gamma$ is infinitely divisible if and only if $-\log\gamma$ is of conditionally negative type (Corollary~1.6 in~\cite{Horn69}). Applying this to $\gamma=1/\beta$, we come back again to the fact that $\log\beta$ is  of conditionally negative type.


\bibliographystyle{amsplain}
\bibliography{../../BIB/ma_bib}

\end{document}